\documentclass[review,12pt]{iopart}
\usepackage{xcolor}
\usepackage{amsthm}
\usepackage{amssymb}
\usepackage{lipsum}
\usepackage{amsfonts}
\usepackage{graphicx}
\usepackage{epstopdf}
\ifpdf
  \DeclareGraphicsExtensions{.eps,.pdf,.png,.jpg}
\else
  \DeclareGraphicsExtensions{.eps}
\fi
\usepackage[algo2e]{algorithm2e} 
\newtheorem{thm}{Theorem}[section]
\newtheorem{cor}[thm]{Corollary}
\newtheorem{theorem}[thm]{Theorem}
\newtheorem{prop}[thm]{Proposition}

\newtheorem{rem}[thm]{Remark}
\newcommand{\commentout}[1]{}

\newcommand{\nwc}{\newcommand}

\nwc{\nwt}{\newtheorem}
\nwc{\bR}{\mb R}
\nwc{\bH}{{\mb H}}
\nwc{\bxp}{{{\mathbf x}}}
\nwc{\bap}{{{\mathbf y}}}

\nwc{\bPhi}{\mathbf{\Phi}}
\nwc{\bPsi}{\mathbf{\Psi}}
\nwc{\bh}{\mathbf h}

\nwc{\bI}{\mathbf I}
\nwc{\bP}{\mathbf P}
\nwc{\bd}{\mathbf{d}}
\nwc{\bX}{\mathbf X}

\nwc{\om}{\omega}
\nwc{\xp}{{x^{\perp}}}
\nwc{\yp}{{y^{\perp}}}
\nwc{\ba}{{\mb a}}
\nwc{\bal}{\begin{align}}
\nwc{\ben}{\begin{equation*}}
\nwc{\beqq}{\begin{equation}}
\nwc{\bea}{\begin{eqnarray}}
\nwc{\beq}{\begin{eqnarray}}
\nwc{\bean}{\begin{eqnarray*}}
\nwc{\beqn}{\begin{eqnarray*}}
\nwc{\beqast}{\begin{eqnarray*}}

\nwc{\eal}{\end{align}}
\nwc{\een}{\end{equation*}}
\nwc{\eeqq}{\end{equation}}
\nwc{\eea}{\end{eqnarray}}
\nwc{\eeq}{\end{eqnarray}}
\nwc{\eean}{\end{eqnarray*}}
\nwc{\eeqn}{\end{eqnarray*}}
\nwc{\eeqast}{\end{eqnarray*}}

\nwc{\vep}{\varepsilon}
\nwc{\ep}{\epsilon}
\nwc{\ept}{\epsilon}
\nwc{\vrho}{\varrho}
\nwc{\orho}{\bar\varrho}
\nwc{\ou}{\bar u}
\nwc{\vpsi}{\varpsi}
\nwc{\lamb}{\lambda}
\nwc{\Var}{{\rm Var}}
\nwc{\nn}{\nonumber}
\nwc{\mf}{\mathbf}
\nwc{\mb}{\mathbf}
\nwc{\ml}{\mathcal}
\nwc{\IA}{\mathbb{A}} %algebraic
\nwc{\bo}{\mathbf o}
\nwc{\IB}{\mathbb{B}}
\nwc{\IC}{\mathbb{C}} %complex
\nwc{\ID}{\mathbb{D}} %Dedekind
\nwc{\IM}{\mathbb{M}} %Dedekind
\nwc{\II}{\mathbb{I}} %Dedekind
\nwc{\IE}{\mathbb{E}} %uklides
\nwc{\IF}{\mathbb{F}} %finite field
\nwc{\IG}{\mathbb{G}} %Gauss
\nwc{\IN}{\mathbb{N}} %natural
\nwc{\IQ}{\mathbb{Q}} %rational
\nwc{\IR}{\mathbb{R}} %real
\nwc{\IT}{\mathbb{T}} %torus
\nwc{\IZ}{\mathbb{Z}} %integers

\nwc{\cE}{{\ml E}}
\nwc{\cP}{{\ml P}}
\nwc{\cQ}{{\ml Q}}
\nwc{\cL}{{\ml L}}
\nwc{\cX}{{\ml X}}
\nwc{\cW}{{\ml W}}
\nwc{\cZ}{{\ml Z}}
\nwc{\cR}{{\ml R}}
\nwc{\cV}{{\ml V}}
\nwc{\cT}{{\ml T}}
\nwc{\crV}{{\ml L}_{(\delta,\rho)}}
\nwc{\cC}{{\ml C}}
\nwc{\cO}{{\ml O}}
\nwc{\cA}{{\ml A}}
\nwc{\cK}{{\ml K}}
\nwc{\cB}{{\ml B}}
\nwc{\cD}{{\ml D}}
\nwc{\cF}{{\ml F}}
\nwc{\cS}{{\ml S}}
\nwc{\cM}{{\ml M}}
\nwc{\cG}{{\ml G}}
\nwc{\cH}{{\ml H}}
\nwc{\bk}{{\mb k}}
\nwc{\bn}{{\mb n}}
\nwc{\cbz}{\overline{\cB}_z}
\nwc{\supp}{{\hbox{supp}}}
\nwc{\fR}{\Re}
\nwc{\bY}{\mathbf Y}

\nwc{\pft}{\cF^{-1}_2}
\nwc{\bU}{{\mb U}}
\nwc{\bG}{{\mb G}}
\nwc{\bg}{\mathbf{g}}
\nwc{\mbf}{\mathbf{f}}
\nwc{\mbe}{\mathbf{e}}
\nwc{\be}{\mathbf{e}}
\nwc{\Om}{\Omega}
\nwc{\ind}{\operatorname{I}}
\nwc{\mbx}{\mathbf{f}}
\nwc{\bb}{\mathbf{g}}
\nwc{\xmax}{f_{\rm max}}
\nwc{\xmin}{f_{\rm min}}
\nwc{\suppx}{\hbox{\rm supp} (\mbf)}
\nwc{\by}{\mathbf{h}}
\nwc{\bZ}{\mathbf{Z}}
\nwc{\bF}{\mathbf{F}}
\nwc{\bE}{\mathbf{E}}
\nwc{\bV}{\mathbf{V}}
\nwc{\cI}{\IZ^2_N}
\nwc{\chis}{{\chi^{\rm s}}}
\nwc{\chii}{{\chi^{\rm i}}}
\nwc{\pdfi}{{f^{\rm i}}}
\nwc{\pdfs}{{f^{\rm s}}}
\nwc{\pdfii}{{f_1^{\rm i}}}
\nwc{\pdfsi}{{f_1^{\rm s}}}
\nwc{\thetatil}{{\tilde\theta}}
\nwc{\red}{\color{red}}
\nwc{\prox}{\hbox{prox}}
\nwc{\sloc}{J_{\rm f}}
\nwc{\bu}{\xi}
\nwc{\bv}{\eta}
\nwc{\cU}{\mathcal{U}}
\nwc{\cN}{\mathcal{N}}
\nwc{\bN}{\mathbf{N}}
\nwc{\mbm}{\mathbf{m}}
\nwc{\bw}{\mathbf{w}}
\nwc{\im}{i}
\nwc{\bom}{\mathbf{w}}
\nwc{\bt}{\mathbf{t}}
\nwc{\z}{y}
\nwc{\cY}{\mathcal{Y}}
\usepackage{amsopn}
\DeclareMathOperator{\diag}{diag}
\begin{document}

\title[Local saddles of RAAR algorithms]{
Local saddles  of  relaxed averaged alternating reflections algorithms on phase retrieval
}

\author{Pengwen Chen}

\address{Applied mathematics, National Chung Hsing University, Taiwan}
\ead{pengwen@nchu.edu.tw}
\vspace{10pt}
\begin{indented}
\item[]%August 2017
\end{indented}

\begin{abstract}
Phase retrieval can be expressed as  a non-convex constrained optimization problem to identify one phase minimizer one a torus.   Many iterative transform techniques have been proposed to identify the minimizer, e.g., relaxed averaged alternating reflections(RAAR) algorithms. In this paper, we present one optimization viewpoint on the RAAR algorithm.
    RAAR algorithm is  one alternating direction method of multipliers(ADMM) with one penalty parameter.
    Pairing  with  multipliers (dual vectors), phase vectors on the primal space are lifted to 
   higher dimensional vectors, RAAR algorithm  is one continuation algorithm, which  searches for  local saddles  in  the primal-dual space.   The dual iteration approximates one gradient ascent flow, which   drives the corresponding   local minimizers    in  a positive-definite   Hessian region. 
            Altering  penalty  parameters,   the RAAR  avoids the stagnation of   these corresponding local  minimizers in the primal space and thus
       screens out  many stationary points corresponding to  non-local minimizers. 
\end{abstract}

%
% Uncomment for keywords
\vspace{2pc}

\section{Introduction}Phase retrieval has recently attracted  attentions  in the mathematics community (see one review \cite{Shechtman2015} and references therein).  The problem of phase retrieval is motivated by  the  inability  of  photo detectors  to  directly  measure  the  phase  of  an  electromagnetic  wave  at  frequencies  of THz  (terahertz)  and  higher. 
The problem of phase retrieval aims to reconstruct
an unknown object $x_0\in \IC^n$ from its magnitude measurement data $b=|A^* x_0|$, where 
$A\in \IC^{n\times N} $ represents  some  isometric matrix and  $A^*$ represents the Hermitian adjoint of $A$. Introduce one
 non-convex $N$-dimensional torus associated with  its normalized torus  $ \cZ:=\left\{z\in \IC^N: |z|=b\right\},\; \cU:=\left\{u\in \IC^N: |u|=1\right\}. $
The whole problem is equivalent to  reconstructing   the missing phase information $u$ and the unknown object $x=x_0$ via solving  the constrained least squares  problem
\beqq\label{main_P}
\min_{x\in \IC^n, |u|=1} \left\{ \|b\odot u-A^* x\|^2: u\in \IC^N\right\}=
\min_{z\in \cZ} \|A_\bot z\|^2,
\eeqq
where 
 $A_\bot\in \IC^{(N-n)\times N}$
is  an isometric matrix  with unitary matrix $[A^*, A_\bot^*]$,  \[
[A^*, A_\bot^*][A^*, A_\bot^*]^*=A_\bot^* A_\bot+A^* A=I.\]
Here, $b\odot u$ represents the component-wise multiplication
between two vectors $b,u$, respectively.
The isometric condition is not very  restrictive in applications, since Fourier transforms are commonly applied in phase retrieval. Even for non-Fourier transforms, we can still obtain equivalent problems via  a QR-factorization, see \cite{Chen2017}.
Let $\cU_*$ denote the set in $\cU$ consisting of all the local minimizers of (\ref{main_P}).
A vector $z_*\in \cZ$ minimizes (\ref{main_P})
is called a global solution.
In the noiseless measurement case, $A_\bot z_*=0$ or $z_*=A^* x_*=b\odot u_*$ for some $u_*\in \cU$ and some $x_*\in \IC^n$.   Numerically,  it is a  nontrivial task  to  obtain   a global  minimizer  on the non-convex torus. 
   The error reduction method is one traditional  method~\cite{Gerchberg1972}, which could  produce a local solution of poor quality   for (\ref{main_P}), if no proper initialization is taken. During last decades, researchers propose various spectral initialization  algorithms  to overcome this challenge\cite{NetrapalliEtAl2013Phase,ChenCandes2015Solving,CandesEtAl2015Phase, NIPS2016_83adc922, Chen2017,Chen2017a, LuoEtAl2018Optimal,LuLi2017Phase,Mondelli2019, duchi2019solving}. 
On the other hand,  phase retrieval can be also tackled by another class of algorithms, including
the well-known  hybrid input-output algorithm(HIO)\cite{Fienup1982, Fienup2013}, the hybrid projection–reflection method\cite{Bauschke2003},  Fourier Douglas-Rachford algorithm (FDR)\cite{CHEN2018665}, alternating direction methods\cite{Wen_2012} and relaxed averaged alternating reflections(RAAR) algorithms\cite{Luke_2004}.  An important feature of these algorithms is the  empirical ability  to  avoid  local  minima  and  converge  to  a  global  mini-mum for noise-free oversampled diffraction patterns.  For instance, 
the empirical study of  FDR indicates the disappearance of  the stagnation at poor local  solutions
under  sufficiently many random masks.
A limit point of FDR  is a global solution in (\ref{main_P}) and  the limit
 point  with  appropriate spectral gap  conditions  reconstructs    the phase retrieval solution~\cite{CHEN2018665}.
Traditional  convergence study on  Douglas-Rachford  splitting algorithm~\cite{Eckstein1992,He2015} heavily relies on the convexity assumption.
%   Without convexity assumptions, 
% convergence   to a stationary point can be  established instead  under  large penalty parameters~\cite{Li2016}. 
        Noise-free measurement is a strict requirement for HIO and FDR,  which motivates the proposal of
  relaxed averaged alternating reflections algorithm~\cite{Luke_2004, Li}. 
Let $\cA,\cB$ denote the sets $Range(A^*)$ and $\cZ$, respectively. Let $P_\cA$ and $P_\cB$ denote the projector on $\cA$ and $\cB$, respectively.  Let $R_\cA, R_\cB$ denote the reflectors corresponding to $\cA,\cB$.    
   With one parameter $\beta\in (0,1)$  relaxing the original feasibility problem (the intersection of $\cA$ and $\cB$), the $\beta$-RAAR algorithm \cite{Luke_2004} is defined as the iterations $\left\{S^k(w): k=1,2,\ldots \right\}$ for some initialization $w\in \IC^N$,
\begin{eqnarray}
S(w)&=& \beta \cdot \frac{1}{2} (R_\cA R_\cB+I) w+(1-\beta) P_\cB w\\
&=& \frac{\beta}{2}\{(2A^*A-I) (2b\odot \frac{w}{|w|}-w)+w\}+(1-\beta )b\odot \frac{w}{|w|}
\\
&=&{\beta} w+(1-2{\beta}) b\odot \frac{w}{|w|}+{\beta} A^* A (2b\odot \frac{w}{|w|}-w).\label{RAART}
\end{eqnarray}
 Fourier Douglas-Rachford  algorithm can be deemed as an extreme case of 
$\beta$-RAAR family with  $\beta=1$.
As RAAR converges to a fixed point $w$, we could retrieve the phase information $u=w/|w|$ for (\ref{main_P}).
Any $u$ in $\cU_*$ yields a fixed point $w$. 
Empirically, RAAR fixed points  can produce local solutions of high quality, if a large value is properly chosen for $\beta$, as reported in~\cite{Wen_2012,Li}. 

In this work, we disclose the relation between RAAR  and the local minimizers $z$ in (\ref{main_P}). As HIO can be reformulated as one  alternating direction method of multipliers in \cite{Wen_2012}, 
we identify RAAR as one   ADMM with \textbf{ penalty parameter} $1/\beta'=(1-\beta)/\beta$ applied to the  constrained optimization problem in~(\ref{main_P}), 
 e.g.,  Theorem.~\ref{ADM_RAAR}. 
This perspective links   $\beta$-RAAR with  a small parameter $\beta$  to  multiplier methods with 
 large penalty  ${\beta'}^{-1}$. 
It is known in optimization that convergence of  a multiplier method
  relies on  
a sufficiently large penalty (e.g., see Prop. 2.7  in \cite{Bertsekas1996}). From this perspective, it is not surprising that the convergence of  RAAR  to its fixed point also requires  a large penalty parameter.
Actually, 
large penalty  has been employed to  ensure various ADMM iterations   converging  to   stationary points~\cite{Hong2016, li2015global, wang2019global}. For instance, ADMM \cite{wang2019global} is applied to solve the  minimization of nonconvex nonsmooth functions. Global convergence to a stationary point can be established, when sufficiently large penalty parameters are used.

 Saddle  plays a fundamental role in the theory and the application of  convex optimization\cite{Bert},  in particular, the convergence of ADMM, e.g., \cite{Boyd}.  For the application  on phase retrieval,
  Sun et al\cite{Sun}  conduct saddle analysis on a quatradic  objective function of  Gaussian measurements.
The geometric analysis  shows that with high probability the global solution is the one local minimizer, when $N/n$ is sufficiently large.
Most of critical points are saddles at actually.
We believe that saddle analysis    is also one key ingredient in  explaining    the  avoidance of  undesired critical points for the Lagrangian of RAAR. 
To some extent,  promising empirical  performance of non-convex ADMM   conveys the impression that saddles   exist in the Lagrangian function, which is  not evident   in the context of phase retrieval. This is a motivation of the current study. 
  Recently, researchers have been cognizant of  the importance of   saddle structure in    non-convex optimization research. 
 Analysis of critical points in non-concave-convex problems  leads to many interesting results in various applications. For instance,   Lee et al. used  a dynamical system approach to   show that 
many gradient-descent algorithms almost surely converge  to  local minimizers with random initialization, even though
  they can get stuck at critical points theoretically~\cite{Lee2016,NIPS2017_f79921bb, jin2017escape}. 
  The terminology ``local saddle" is  a crucial concept in understanding the min-max algorithm   employed  in modern machine learning research, e.g.,  gradient descent-ascent algorithms in
generative adversarial networks (GANs)\cite{goodfellow2020generative} and multi-agent reinforcement learning\cite{omidshafiei2017deep}. With proper Hessian adjustment,  
\cite{Adolphs2018} and \cite{Daskalakis2018} proposed
novel saddle algorithms to  escape undesired critical points and to reach  local saddles of min-max problems almost surely with random initialization. 
Jin et al.~\cite{Jin2020} proposed one non-symmetric definition of local saddles to address one basic question, `` what is a proper definition of local optima for the local saddle?" Later,  Dai and Zhang gave  saddle analysis on the constraint minimization problems~\cite{Dai2020}.

 Our study starts with 
one characterization of all the fixed point of RAAR algorithms in Theorem~\ref{fixed1}. These fixed points are critical points of (\ref{main_P}). By varying $\beta$, some of the fixed points  become ``local saddles" of a concave-nonconvex function $F$,  \beqq\label{eq_3_}
\max_\lambda \min_z \left\{F(z,\lambda; \beta):=\left(\frac{\beta}{2} \|A_\bot(z-\lambda)\|^2-\frac{1}{2} \|\lambda\|^2 \right), \; z\in \cZ, \;  \lambda\in \IC^N\right\}.
\eeqq
  To characterize RAAR iterates, 
we investigate saddles in (\ref{eq_3_}) lying in a   high dimensional primal-dual space. Our study aims to  answer   a few  intuitive  questions,  whether  these local dimensional critical points on $\cZ$ in    the primal space can be  lift  to  local saddles of (\ref{eq_3_}) in a primal-dual space under some spectral gap condition,  and how  the ADMM iterates   avoid or converge to these local saddles under a proper  penalty parameter?  
The line of thought motivates the current study on  local saddles of (\ref{eq_3_}).
Unfortunately, the definition of  local saddles in \cite{Jin2020} can not be employed to analyze the RAAR convergence, since the objective function in phase retrieval shares   phase invariance, see Remark~\ref{alpha-79}.
 
The main  goal of the present work is to establish 
an optimization view to illustrate the convergence of RAAR, and 
show by analysis and numerics,
under the framework for phase retrieval with coded diffraction patterns in ~\cite{Fannjiang_2012}, RAAR has a basin of attraction at a local saddle $(z_*,\lambda_*)$.
For noiseless measurement, $z_*=A^* x_0$ is a strictly local minimizer of  (\ref{main_P}).
In practice,  numerical stagnation of RAAR on noiseless measurements disappears under sufficient large $\beta$ values.
 Specifically, Theorem~\ref{ADM_RAAR} shows that 
 RAAR is actually   one ADMM to solve the constrained problem in (\ref{main_P}). Based on this identification, Theorem~\ref{Main1} show that  each limit of RAAR iterates can be  viewed as a ``local saddle" of $\max\min F$ in (\ref{eq_3_}).

The rest of the paper is organized as follows. In section~\ref{sec:2}, we examine the fixed point condition of RAAR algorithm. By identifying RAAR as ADMM, we  disclose the concave-non-convex function for the dynamics of RAAR, which provides  a continuation viewpoint on RAAR iteration. 
 In section~\ref{sec:3}, we present  one proper definition for local saddles and show
 the existence of local saddles for oversampled coded diffraction patterns.   In section~\ref{sec:4}, we show the convergence of RAAR  to a local saddle  under a sufficiently large parameter.  Last, we provide  experiments to illustrate  the behavior of RAAR, (i)comparison experiments between   RAAR and Douglas Rachford splitting proposed in \cite{fannjiang2020fixed}; (ii) applications of RAAR on coded diffraction patterns.

    \section{RAAR algorithms}\label{sec:2}

    \subsection{Critical points }

   The following gives the first order optimality of the problem in (\ref{main_P}). This is a special case of Prop.~\ref{convex1} with $\lambda=0$. We skip 
    the proof. 
 \begin{prop}  \label{convex0} Let $z_0=b\odot u_0$ be a local minimizer of the problem in (\ref{main_P}). 
Let $ K_{z_0}^\bot:=\Re(\diag(\bar u_0)A_\bot^* A_\bot \diag(u_0))$.
  Then  the first-order optimality condition is
\beqq\label{Ln10}
q_0:={z_0}^{-1}\odot (A_\bot^* A_\bot z_0)\in \IR^N, \eeqq
and the second-order necessary condition is that 
for all $\xi\in \IR^N$,
\beqq\label{Ln20}
\xi^\top (K_{z_0}^\bot-\diag(q_0))\xi\ge 0.
\eeqq
 
 \end{prop}  
 \begin{rem}
{  Once a local solution $z$ is obtained, the unknown object of phase retrieval in (\ref{main_P}) can be  estimated by   $x=Az$. }  
 On the other hand, 
  using $I=A^*A+A_\bot^* A_\bot$, we can express the first order condition  as 
 \beqq
 u^{-1}\odot (A^*A(b\odot u))=b\odot (1-q_0)\in \IR^N.\eeqq
 Using $\xi=e_i$  canonical vectors of $\IR^N$,  we have a componentwise  lower bound on $1-q_0$ from (\ref{Ln20}):
 $b^{-1}\odot (1-q_0)\ge \|A e_i\|^2\ge 0$. In general, there exists many local minimizers on $\cZ$, satisfying (\ref{Ln10}) and (\ref{Ln20}).

 \end{rem}
  \subsection{Fixed point conditions of RAAR}We begin  with fixed point conditions of $\beta$-RAAR iterations in~(\ref{RAART}). 
   For each $\beta\in (0,1)$,  introduce one auxiliary  parameter $\beta'\in (0,\infty)$ defined by  \beqq\label{beta'}
\beta=\frac{\beta' }{1+\beta'},\; i.e.,  \beta'=\frac{\beta}{1-\beta}.
\eeqq
 We shall  show the reduction in the cardinality of fixed points under  with a  small penalty parameter.
\begin{theorem}\label{fixed1}
Consider the application of the $\beta$-RAAR algorithm on the constrained problem in (\ref{main_P}).
 Write $w\in \IC^N$ in polar form $w=u\odot |w|$. For each $\beta\in (0,1)$, let ${\beta'}={\beta}/(1-{\beta})$ and
  \beqq \label{cueq} c:= (1-\frac{1-{\beta}}{{\beta}}) b+\frac{1-{\beta}}{{\beta}} |w|\in \IR^N.
  \eeqq
 Then $w$ is a fixed  point of $\beta$-RAAR, if and only if 
 $w$ satisfies  the  phase condition
\beqq\label{f3}
 A^*A(b\odot u)=c\odot  u,\; 
\eeqq
 and   the  magnitude condition,
 \beqq\label{absy}
|w|={\beta'} c+ (1-{\beta'})b\ge 0, i.e.,  \; c\ge (1-{\beta'}^{-1}) b.
\eeqq
In particular, for 
 ${\beta} \in [1/2, 1)$, we have $c\ge 0$ from (\ref{absy}). Observe that  the inequality in (\ref{absy}) ensures  the well-defined magnitude vector   $|w|$. Hence, the fixed points are critical points of (\ref{main_P}).
\end{theorem}

\begin{proof}
Rearranging  (\ref{RAART}),  we obtain the fixed point condition of RAAR, 
\beqq\label{RAAR_fix}
((1-{\beta})|w|-(1-2{\beta})b)\odot \frac{w}{|w|}={\beta} A^*A \left\{ (2b-|w|)\odot \frac{w}{|w|}\right\}.
\eeqq
Equivalently, taking the projections $A^*A$ and $I-A^*A$  on (\ref{RAAR_fix}) yield 
\begin{eqnarray}&& A^* A \left\{ (b-|w|) \odot \frac{w}{|w|}\right\}=0, \label{f1}\\
&& (I-A^*A)\left\{ b\odot \frac{w}{|w|}\right\}= {\beta'}^{-1}\left\{(b-|w |)\odot \frac{w}{|w|}\right\}.\label{f2}
\end{eqnarray}
 
For  the only-if part, let  $w$ be a fixed point of RAAR with (\ref{f1},\ref{f2}).  
With the definition of $c$,  (\ref{f1}) gives   \beqq\label{part1_0}
A^*A ((c-b)\odot u)={\beta'}A^*A((|w|-b)\odot u)=0,
\eeqq
and  (\ref{f2}) gives
\beqq
A_\bot^* A_\bot (c\odot \frac{w}{|w|})=A_\bot^* A_\bot \left(\left\{ b-{\beta'}^{-1} (b-|w|)\right\}\odot \frac{w}{|w|}\right)=0,\eeqq
 which implies $ c\odot u$
 in the range of $A^*$. Together with  (\ref{part1_0}), we have  (\ref{f3}). Also, 
  (\ref{absy}) is the result of the non-negativeness of $|w|$ in (\ref{cueq}).
   
To verify the if-part, we need to show  that $w$ constructed from a phase vector $u\in \IC^N$
 satisfying (\ref{f3}) 
 and a magnitude vector $|w|$ satisfying 
   (\ref{absy})
meets  (\ref{f1},\ref{f2}).
From (\ref{cueq})
and (\ref{f3}), we have (\ref{f1}), i.e., 
\beqq
A^*A((b-|w|)\odot u)=A^*A\left\{ \beta'(b-c)\odot u\right\}={\beta'}\left\{c\odot u-c\odot u\right\}=0.
\eeqq
With the aid of (\ref{f1}, \ref{f3}),  the fixed point condition in (\ref{f2}) is ensured  by the computation: \beqq
(I-A^*A)\left\{(b-{\beta'}^{-1}(b-|w|))\odot u\right\}
=(I-A^*A)\left\{c\odot u\right\}=0. \eeqq

Finally,  
 the condition in (\ref{f3}) is identical to the first optimality condition in (\ref{Ln10}). Hence,  the fixed points
 must be critical points of (\ref{main_P}).
$\square$
\end{proof}
Theorem~\ref{fixed1} indicates that each fixed point $w$
 can be   re-parameterized  by  $(u,\beta)$ satisfying (\ref{f3}) and (\ref{absy}). The condition in (\ref{absy}) always hold 
for  ${\beta'}$  sufficiently small.

{ 
\begin{rem}
The first order optimality { in (\ref{Ln10})} yields that the phase condition in (\ref{f3}) is actually the critical point condition of $u\in \cU_*$ in (\ref{main_P}). Fix one critical point $u\in \cU_*$ and let $c$ be the corresponding vector given   from  (\ref{f3}).  
 From Theorem~\ref{fixed1}, 
 $w$ given from the polar form $w=u\odot |w|$ with (\ref{absy}) is a fixed point 
 of $\beta$-RAAR,  if { $\beta$ satisfies the condition in (\ref{absy}). To further examine  (\ref{absy}),} we parameterize the fixed point $w$ by $(u,\beta)$. Let
$b^{-1}\odot K^\bot b$ denote the  threshold vector, where
 \[ K:=\Re(\diag(\bar u)A^*A \diag(u)), \; K^\bot:=I-K.
 \] The fixed point condition in  (\ref{absy}) indicates that 
 $(u,\beta_1)$  gives 
 a fixed point of $\beta_1$-RAAR with any ${\beta}_1\in (0,{\beta})$.  { That is, 
 the corresponding   parameter  $(\beta')^{-1}$ must exceed the threshold vector,  \beqq
 \label{eq_pc}
(\beta')^{-1}=\frac{1-\beta}{\beta}\ge b^{-1}\odot (K^\bot b).
\eeqq    Since  $\beta'= \beta/(1-\beta)$ can be viewed as   one penalty parameter in the associated Lagrangian in  (\ref{Leq''}), 
we  call (\ref{eq_pc}) the \textbf{penalty-threshold condition } of RAAR fixed points.  In general,
 the cardinality of RAAR fixed points
decreases  under
 a large parameter $\beta$. See Fig.~\ref{beta_interval}.}
\end{rem}
}

  For ${\beta}=1$, RAAR reduces to   FDR, whose fixed point $w$ satisfies $\|A_\bot (b\odot w/|w|)\|=0$ and thus $\|A(b\odot w/|w|)\|=\|b\|$. When phase retrieval has uniqueness property, $A(b\odot w/|w|)$ gives the reconstruction. On the other hand, for  $\beta=1/2$, 
 (\ref{RAART}) gives
\beqq\label{eq6}
S(w)=A^*A (b\odot \frac{w}{|w|})+\frac{1}{2}(I-A^*A) w.
\eeqq
Suppose   a RAAR initialization is chosen from  the range of   $A^* $. The second term in (\ref{eq6}) always vanishes and thus  RAAR  iterations reduce to     alternating projection iterations(AP) in~\cite{Chen2017}. From this perspective, 
one can regard $\beta$-RAAR as one family of algorithms interpolating AP and FDR,  varying  $\beta$ from $1/2$ to $1$.  

  \begin{figure}
 \begin{center}
     \includegraphics[width=0.4\textwidth]{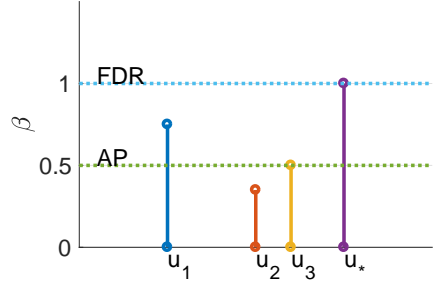} 
          \end{center}
   \caption{  
{ 
 { Illustration of  \textbf{the penalty-threshold  condition } of RAAR fixed points associated with critical points $u_1,u_2, u_3, u_*\in \cU_*$ of (\ref{main_P}). The set of  fixed points of $\beta$-RAAR  is a collection of line segments parameterized by $(u,\beta)\in \cU_*\times (0,1)$.  As $\beta$ gets larger,  the cardinality of intersections (i.e., the fixed points) decreases.     Critical points associated with  $\beta=0.5$ are   $u_1, u_3$ and $ u_*$.  The global minimizer $u_*$ is the only  associated critical point, if   $\beta\in (0.9,1)$ is used.}}}\label{beta_interval}
    \end{figure}

\subsection{  Alternative directions method of multipliers}
Next, we present one relation  between RAAR and the  alternating direction method of multipliers (ADMM). 
The alternating direction method of multipliers  was originally introduced in the 1970s\cite{Glowinski1975,Gabay1976} and can be regarded as  an approximation of the augmented Lagrangian method, whose  primal update step is replaced by one iteration of the alternating minimization.
Although ADMM is classified as one first-order method, practically ADMM could produce a solution with modest accuracy  within a reasonable amount of time. Due to the algorithm simplicity,   nowadays  this approach is  popular in many applications, in particular, applications of nonsmooth optimalication.  See~\cite{Boyd,  yang2009, Deka2019}  and the references
therein.  
%  equivalently 
% \beqq\label{P2} \min_{z\in\cZ} \|A_\bot z\|^2,\; 
% \cZ:=\left\{z\in\IC^N: |z|=b\right\}.
%\eeqq

Use the standard inner product
\[
\langle x,  y\rangle:=\Re(x^* y),\; \forall x,y\in\IC^N.
\] 
To solve  the problem in (\ref{main_P}), introduce one auxiliary  variable $y\in \IC^N$ with one constraint $y=z$ and one associated  multiplier $\lambda\in \IC^N$,  
and form
 the Lagrangian function
  with some parameter $\beta'>0$ in (\ref{beta'}),
   \beqq\label{Leq}
 \frac{\beta'}{2} \|A_\bot y\|^2+\left< \lambda, y-z\right>+\frac{1}{2}\|y-z\|^2, \; y\in \IC^N, z\in \cZ.
\eeqq
Equivalently, we have the augmented Lagrangian function,
\beqq\label{Leq''}
 L(y,z, \lambda):= \frac{1}{2} \|A_\bot y\|^2+
 {\beta'}^{-1}\left< \lambda,  y-z\right>+\frac{1}{2\beta'}\|y-z\|^2,
\eeqq when we 
identify
 $1/\beta'$ as
a \textbf{penalty parameter}.
To solve $(y,z,\lambda)$, ADMM
starts with some initialization $z_1\in \cZ$ and $ \lambda_1\in \IC^N$ and generates the sequence $\left\{(y_k, z_k, \lambda_k): k=1,2,3,\ldots\right\}$ with stepsize $s>0$, according to rules, 
 \begin{eqnarray*}
&&
y_{k+1}=arg\min_y L(y,z_k, \lambda_k),\label{y1}\\
&&
z_{k+1}=arg\min_{|z|=b} L(y_{k+1},z, \lambda_k)\label{z1},
\\
&&
\lambda_{k+1}=\lambda_k+s\nabla_\lambda L(y_{k+1}, z_{k+1}, \lambda)\label{eq9}
\end{eqnarray*}
Introducing one projection operator  on $\cZ$,  $[w]_{\cZ}:=w/|w|\odot b$ for $w\in \IC^N$. Algebraic computation yields
 \begin{eqnarray}
&&
y_{k+1}=(I+{\beta'} A_\bot ^*A_\bot )^{-1} (z_k-\lambda_k)=(I-{\beta} A_\bot ^* A_\bot )(z_k-\lambda_k), \label{y}
\\
&&
z_{k+1}=[y_{k+1}+\lambda_k]_\cZ,\label{z}
\\
&&\label{la11}
\lambda_{k+1}=\lambda_k+s(y_{k+1}-z_{k+1}).
\end{eqnarray}
From the $y$-update in (\ref{y}), one reconstruction $x$ for the unknown object  in (\ref{main_P}) can be computed by
\beqq\label{eq_x}
x=Ay=A(I-\beta A_\bot^* A_\bot)(z-\lambda)=A(z-\lambda).
\eeqq

Theorem~\ref{ADM_RAAR} indicates  that    RAAR is actually an ADMM  with proper initialization applied to   the  problem
 in (\ref{main_P}). In general, the step size $s$ of the dual vector $\lambda$ should be chosen properly to ensure the convergence. The following shows the relation between RAAR and the  ADMM with $s=1$. Hence, we shall focus  $s=1$ in this paper. 

%Proposition~\ref{ADM_RAAR} demonstrates that RAAR iterations in (\ref{RAART}) 
% are actually  the ADMM iterations   to compute ``local saddles" associated with the  max-min problem
%\beqq\label{L1}
%\max_{\lambda} 
% \min_{z\in \cZ}\min_{y\in \IC^N}  L(y,z, \lambda).
%\eeqq
%Here, the precise definition on ``local saddles" would be discussed in section \ref{sec:3}.

\begin{theorem}\label{ADM_RAAR}
Consider one $\beta$-RAAR iteration $\{{  w_0,w_1,}\ldots, \}$ with nonzero initialization  ${  w_0}\in \IC^N$.
 Let  $
\lambda_1=A_\bot^* A_\bot w_0, \; z_1=[w_0]_\cZ.
$
Generate 
 one ADMM sequence $\left\{(y_{k+1}, z_{k+1}, \lambda_{k+1}): k={  1,2,}\ldots\right\}$ with dual step size $s=1$, according to (\ref{y}, \ref{z}, \ref{la11}) with initialization 
{   $
\lambda_1, z_1$.}
Construct a sequence $\left\{w'_k: k=1,2,\ldots\right\}$ from $(y_k, z_k, \lambda_k)$,  
\beqq\label{2.18} w'_k:=y_{k+1}+\lambda_k=(I-\beta A_\bot^* A_\bot) z_k+\beta A_\bot^* A_\bot\lambda_k.
\eeqq  Then
 the sequence $\left\{w'_k: k=1,2,\ldots\right\}$ is exactly the $\beta$-RAAR sequence, i.e., $w'_k=w_k$ for all $k\ge 1$.
 
\end{theorem}

\begin{proof} Use induction. 
For $k=1$, we have\[
w_1'=(I-\beta A_\bot^* A_\bot) z_1+\beta A_\bot^* A_\bot\lambda_1=
(I-\beta A_\bot^* A_\bot) [w_0]_\cZ+\beta A_\bot^* A_\bot w_0=w_1.
\]
Suppose $w_k'=w_k$ for $k\ge 1$. 
From (\ref{2.18}) and (\ref{z}), we have \beqq\label{eq9_1}
A_\bot^* A_\bot w'_k=A_\bot^* A_\bot((1-\beta ) z_k+\beta \lambda_k),
\eeqq
and  $ z_{k+1}=[w_k']_\cZ$.
From (\ref{la11}) and (\ref{y}), 
\beqq\label{2.21}
A_\bot^* A_\bot \lambda_{k+1}=A_\bot^* A_\bot \left\{ \beta  \lambda_k +(1-\beta)  z_k\right\}
-A_\bot^* A_\bot z_{k+1}.
\eeqq
Together with (\ref{eq9_1}), (\ref{z}) and (\ref{2.21}), we complete  the proof by the calculation, 
\begin{eqnarray}
w'_{k+1}&=&
(I-\beta A_\bot^* A_\bot) z_{k+1}+\beta A_\bot^* A_\bot\lambda_{k+1}
\\
&=&(I-\beta A_\bot^* A_\bot ) z_{k+1}+
\beta A_\bot^* A_\bot \left\{ \beta  \lambda_k +(1-\beta)  z_k\right\}
-\beta A_\bot^* A_\bot z_{k+1}\\
&=&(I-2\beta A_\bot^* A_\bot ) [w'_{k}]_\cZ +
\beta
 A^*_\bot A_\bot  w'_k\\
 &=&(I-2\beta A_\bot^* A_\bot ) [w_{k}]_\cZ +
\beta
 A^*_\bot A_\bot  w_k=w_{k+1}.\label{eq12}
\end{eqnarray}
\end{proof}

Theorem ~\ref{ADM_RAAR} provides one max-min viewpoint to explore   the  dynamics of RAAR, which motivates the study in the next section. Indeed, 
after eliminating $y$ in $L$ in (\ref{Leq})  via  (\ref{y}), we end up with a   max-min problem of  a concave-non-convex function $F$,
\beqq\label{rest_F} 
F(z,\lambda; \beta):=\left(\frac{\beta}{2} \|A_\bot(z-\lambda)\|^2-\frac{1}{2} \|\lambda\|^2 \right), \; z\in \cZ, \;  \lambda\in \IC^N.
\eeqq
We can convert RAAR convergence  to its fixed points into  the convergence to     saddles of the function $F$ in (\ref{rest_F}) by one primal-dual algorithm. 
For notation simplicity, we shall omit  $\beta$  in the function $F(z,\lambda; \beta)$, i.e., write $F(z, \lambda)$, if no confusion occurs in the context. 

\section{ Definition of local  saddles for RAAR}\label{sec:3} 

When the objective function $F$ of a max-min problem is not concave-convex, saddle points do not exist in general.  
For some smooth function $F$,
a point $(\lambda,z)$ is said to be 
a local max-min point  in~\cite{doi:10.1137/15M1026924,Daskalakis2018}, if $z=z_*$ is a local minimizer and $\lambda=\lambda_*$ is a local maximizer, i.e., \beqq
\label{dpmaxmin}
F(z_*,\lambda)\le F(z_*,\lambda_*)\le F(z,\lambda_*)
\eeqq for $(z,\lambda)$ near $(z_*,\lambda_*)$.
The standard analysis can give the  first-order and second-order characterizations. The existence of a saddle $(z_*, \lambda_*)$ in fact ensures the minimax equality. Indeed, since
$  \min_z F(z,\lambda)\le  \min_z\max_\lambda  F(z,\lambda)
$, then 
$  \max_\lambda \min_z F(z,\lambda)\le  \min_z\max_\lambda  F(z,\lambda)
$. Together with
\beqq \min_z \max_\lambda F(z,\lambda)
\le 
\max_\lambda F(z_*, \lambda)\le F(z_*,\lambda_*)\le \min_z F(z,\lambda_*)\le \max_\lambda \min_z F(z,\lambda),
\eeqq
we have  that $ \min_z \max_\lambda F(z,\lambda)
= \max_\lambda \min_z F(z,\lambda)$ for $(z,\lambda)$ holds near $(z_*,\lambda_*)$.

 In this paper, we shall adopt the idea on ``local max-min"  proposed in~\cite{Jin2020} 
to emphasize the non-symmetric role of $z,\lambda$, i.e.,   $(\lambda,z)$ is said to be a local max-min point, if for any $(\lambda, z)$ near $(\lambda_*, z_*)$ within a distance $\delta>0$, 
\beqq\label{Jinmaxmin}
\max_{z'} \left\{ F(z',\lambda): \|z'-z_*\|\le h(\delta)\right\} \le F(z_*,\lambda_*)\le F(z,\lambda_*)
\eeqq holds
for some continuous function $h: \IR\to \IR$ with $h(0)=0$. That is, the minimizer $z$ is driven by the dual vector $\lambda$ maximizing the objective function  $F$.
Since $F$ in (\ref{rest_F}) is strictly concave in $\lambda$ for $\beta\in (0,1)$, 
according to Prop. 18, 19 and 20\cite{Jin2020}, the first-order condition is $\nabla_z F(z_*,\lambda_*)=0$ and $\nabla_\lambda F(z_*,\lambda_*)=0$, and
%Since our $F$ here is concave  in $\lambda$, 
 the second-order necessary/sufficient condition can be reduced to   $\nabla_{zz} F(z_*,\lambda_*)\succeq 0$ and 
 $\nabla_{zz} F(z_*,\lambda_*)\succ 0$, respectively. In short,  thanks to the $\lambda$-concavity in $F$,  we end up with  the   identical   characterization for   local-max-min points  in~\cite{doi:10.1137/15M1026924, Daskalakis2018}. For simplicity, we shall also call these local max-min points as local saddles.

\subsection{Local saddles}\label{sec:3.2}

From Theorem~\ref{ADM_RAAR}, the RAAR convergence  of $w_k$ to $w_*$ can be regarded as the convergence to a local saddle $(\lambda_k, z_k)$ of $F$ in (\ref{rest_F}). For each $\lambda_k$, $z_k$ is one approximate of the corresponding minimizer of $F(\lambda_k, z)$. 
From~(\ref{dpmaxmin}),   we have the following optimality of $F$ in $\lambda$ and $z$, respectively. 
 Since  $0<\beta<1$, the strict concavity of $F$ in $\lambda$  ensures the uniqueness of $\lambda$ for any vector $z\in \cZ$. We omit the proof of Prop.~\ref{fb1}. 
\begin{prop}\label{fb1} Fix one $z\in \cZ$. 
The maximizer $\lambda$ of $ F(z,\lambda)$  in (\ref{rest_F}) satisfies
\beqq\label{fb}
\lambda=-\beta' A_\bot^* A_\bot z.
\eeqq
Hence,  $\lambda_*=-\beta' A_\bot^* A_\bot z_*$ holds for a  saddle point  $(z_*,\lambda_*)$.
\end{prop}

 %Hence, any critical points $(z,\lambda)$ with a local minimizer $z$ and (\ref{fb}) are saddle points of $F(z,\lambda)$. 

\begin{prop}\label{convex1} 
 Let  $q(z, \lambda)\in \IC^N$ be a vector-valued  function of $z$ and $ \lambda$, \beqq\label{q_def}
q(z,\lambda):=z^{-1}\odot A_\bot^* A_\bot( z-  \lambda).
\eeqq
Fix one  $\lambda$ in (\ref{rest_F}),
and consider the $z$-minimization 
\beqq\label{L'eq}
\min_{z\in \cZ} \left\{F(z, \lambda):=\frac{\beta}{2}\|A_\bot (z- \lambda)\|^2-\frac{1}{2}\| \lambda\|^2\right\}.
\eeqq
When  $z\in\cZ$ is
a  local minimizer of $F(z, \lambda)$, then 
\beqq\label{q_real0}
q(z,  \lambda)\in \IR^N
\eeqq
 and
\beqq\label{H2c0}  
\xi^\top (K_z ^\bot-\diag(\Re(q)))\xi\ge  0, \; \forall \xi\in\IR^N. 
\eeqq
 \end{prop}
\begin{proof}
Let $u=z/|z|$ be the phase vector of a local minimizer $z$. 
Consider the perturbation $z\to z\odot \exp(i\theta)$ in $\cZ$ applied on $F(z, \lambda)$ with  $\theta\in \IR^N$, $\|\theta\|$ near $0$. 
Variation of $F$  can be expressed as  one function  of    tangent vectors $\xi:=b\odot \theta$ , 
$
H(\xi; z)=H(b\odot \theta; z):=\|A_\bot^* A_\bot (z \odot \exp(i\theta) - \lambda)\|^2.
$ 
With  $z=b\odot u$, the Taylor expansion
 \begin{eqnarray*}
 H(\xi; z )&=&H(0; z)+\{2\left< i (b\odot  u)\odot  \theta,A_\bot^* A_\bot (b\odot u- \lambda)  \right>-\left< (b\odot  u) \odot  \theta^2,A_\bot^* A_\bot (b\odot u- \lambda)  \right> \\
&&+\|A_\bot (b\odot u\odot \theta)\|^2\}+o(\|\theta\|^2) 
 \\
 &=&H(0; z)+2\left< i \xi,q\odot b\right> +\left\{-\left< \xi ^2, q \right>+\xi^\top K^\bot \xi\right\}+o(\|\xi\|^2)
 %\label{HessF}
\end{eqnarray*}
implies that 
the first-order  condition for a local minimizer  $z$ is 
 $
q\in \IR^N,
$ and  the second-order condition is  the positive semi-definite condition, \beqq\label{H_2c}
\left< \xi, ( K^\bot -\diag(q))\xi \right> \ge 0.
\eeqq
%Hence, it is a local minimizer if (\ref{q_real0}) and (\ref{H2c0}) hold.
\end{proof}

Unfortunately,  the above  optimality conditions for  $F$ in (\ref{rest_F}) yields  nonexistence of local saddles under the definition in (\ref{dpmaxmin})!
Consider  a noisy case, 
$\|A_\bot z\|>0$ for all $z\in\cZ$. No Nash equilibrium  of $F$ in (\ref{rest_F}) can exist,
since  (\ref{H2c0}) and (\ref{fb}) cannot hold simultaneously   at any stationary  point of $F$.
 Indeed, suppose that $(z_*,\lambda_*)$ is 
 a saddle point with $\|A_\bot z_*\|>0$. The $\lambda$-optimality in  (\ref{fb})
 gives $\left< z_*,  \lambda_*\right>=-\beta' \|A_\bot z_*\|^2<0$. On the other hand,  as $\xi=b$,  (\ref{H2c0}) gives 
 \beqq 
 \label{3.10}
 \xi^\top (K^\bot-\diag(\Re(q)))\xi=\Re(z^* A_\bot^* A_\bot  \lambda)\ge 0.\eeqq
This inconsistency indicates that   
as $\lambda$ tends to $\lambda_*$, the corresponding local  minimizer $z$ does not approach  $z_*$ continuously. 
In the next subsection, we shall  give a proper definition for the  \textit{local}  Nash equilibrium applied on phase retrievel.
Remark~\ref{alpha-79} indicates this difficulty always exists  in the non convex-concave optimization with  phase invariance.
\begin{rem}[phase-invariance]\label{alpha-79}

Consider the problem
 \[
\min_{ \lambda}\max_{z\in \cZ} \left\{ -F(z, \lambda)\right\},
 \]
 where $F(z,\lambda)=F(\alpha z,\alpha\lambda)$ holds for any complex unit $\alpha$.
Suppose that  $(z, \lambda)$ is a local Nash equilibrium, then 
\beqq\label{eq_81}
F(\alpha z, \lambda)\le F(z,\lambda)\le F(z,\alpha \lambda)
\eeqq for any complex unit $\alpha$ near $1$ and $\lambda$ is a local maximizer. 
Then  phase-invariance of $F$ implies  
\beqq\label{eq_82}
F(\alpha z, \lambda)=F( z,\bar \alpha  \lambda)\le F(z,\lambda).
 \eeqq
 Contradiction to (\ref{eq_81}) always occurs,  if the above inequality in (\ref{eq_82}) is strict.
 \end{rem}

\subsection{Cross sections} \label{sec:3.4}

To alleviate the difficulty in Remark~\ref{alpha-79}, we shall  restrict the neighbourhood  $\cU(z_*)$ by slicing the projected torus $A_\bot \cZ$ into cross sections,  such that (\ref{Jinmaxmin}) can hold locally at each critical point. 
   Fix one $z_0\in \cZ$ 
and introduce 
 the set $
\cZ'(z_0):=\left\{z_0\odot \exp(i\theta): \left<  \theta, b^2 \right>=0, \theta\in \IR^N\right\}.
$
 Consider 
  the optimization problem \beqq\label{main_P'}
\min_{z\in \cZ'(z_0)} \left\{\| A_\bot z\|^2\right\}.
\eeqq
Note that one partition  $\cZ=\cup_\alpha \left\{\cZ'(z_0\alpha): |\alpha|=1 \right\}$
 indicates that 
for each $z\in \cZ$, $z\in \cZ'(\alpha z_0)$ holds for some complex unit $\alpha=\exp(i\rho)$, $\rho\in \IR$.
 Indeed, 
let $1_N\in\IR^N$ be a vector whose entries are all $1$.
Since $z, z_0$ both lie in $\cZ$, then \beqq
z\odot z_0^{-1}=\exp( i \delta)=\exp( i (\theta+\rho 1_N))\textrm{ for some $\delta, \theta\in \IR^N$ }
\eeqq
and $\rho:=\|b\|^{-2}\left< \delta, b^2\right>$  with $0=\left< \theta, b^2\right>$.
To proceed, we need a few notations.  At each $z\in \cZ$, introduce a matrix $K^\bot_{z}$ and a tangent plane $\left\{iu\odot \xi: \xi\in \Xi,\; u=z/|z|
 \right\}$ with   
 \[K^\bot_{z}:=\Re\left(\diag\left(\overline{\frac{z}{|z|}}\right) A_\bot^* A_\bot \diag\left(\frac{z}{|z|}\right)\right),
  \Xi:=\left\{ \xi: \xi\in \IR^N, \left<\xi,  b\right>=0\right\}.\]
We shall drop the dependence on $z$ to simplify the notation,  if no confusion occurs. 
%Let  the set  $\left\{\alpha z: |\alpha|=1, \alpha\in \IC\right\}$ be  the  $z$-circle associated with each $z\in \cZ$. 
Since the feasible set in $z$ is different from the setting in Prop.~\ref{convex1}, we must investigate again 
the local  $z$-optimality in $F(z,\lambda)$. 
 
\begin{prop}\label{convex4} 
%Suppose $z^* A_\bot^* A_\bot \lambda\neq 0$. 
%Then we have
%\beqq\label{q_real_la}
%\Im(q(z,  \lambda))=0.
%\eeqq
% In addition, for all $\xi\in \IR^N$ with 
% $\left< \xi, \Re(\overline{u}\odot (A_\bot^* A_\bot  \lambda))\right>=0$,
% we have
%\[
%\left< \xi, (K^\bot-\diag(q))\xi\right>\ge 0.
%\]  \end{prop}
%
%

 Fix one  $\lambda\in \IC^N$.
 % in (\ref{rest_F}).
Consider the minimization problem 
\beqq\label{L''eq'}
\min_{z\in \cZ'(z_0)}  \left\{F(z,  \lambda):=\frac{\beta}{2}\|A_\bot ( z-  \lambda)\|^2-\frac{1}{2}\| \lambda\|^2  \right\}.
\eeqq 
 Suppose $z=b\odot u$ is a local minimizer in (\ref{L''eq'}).  Then $\Im(q(z,  \lambda))=\rho 1_N$, $\rho\in \IR$. 
The second-order necessary condition is that for all $\xi\in \Xi$, we have
\[
\left< \xi, (K_z^\bot -q(z,  \lambda))\xi\right>\ge 0.
\]
In addition, a second-order sufficient condition is that for all nonzero $\xi\in \Xi$, 
\beqq\label{eq_40_}
\|\xi\|^{-2}\left< \xi, (K_z^\bot -q(z,  \lambda))\xi\right>> 0.
\eeqq
A local minimizer  $z$ with (\ref{eq_40_}) is called a \textbf{strictly local minimizer}. 
\end{prop}
\begin{proof}
Consider a perturbation $z\to z\odot \exp(i\theta)$ with $\theta\in \IR^N$. Use arguments  similar to the proof of Prop.~\ref{convex1}. Since the objective function  in (\ref{L''eq'}) is continuously differentiable, with $\xi:=b\odot \theta$, 
 we have
\beqq\label{eq64_1}
\left< \xi, i\overline{u}\odot A_\bot^* A_\bot (z- \lambda)\right>=0
\eeqq
for all $\xi$ with 
\beqq\label{eq64_2}
\left< \xi, b\right>=
\left< \theta, b^2\right>=0.
\eeqq
Then (\ref{eq64_1}) gives for some multiplier $\rho\in\IR$,  
 \beqq\label{eq3.18}
  \Im(\overline{u}\odot A_\bot^* A_\bot  (z- \lambda))=\rho b, \; i.e.,\; \Im(q)=\rho 1_N\in \IR^N. 
 \eeqq
Note that $\xi\in\Xi$ and thus  we have the second-order  conditions. 
$\square$

\end{proof}

\begin{cor}  \label{local_cond} Let $z_0=b\odot u_0$ be a local minimizer of the problem in (\ref{main_P'}).
%, i.e., $\lambda=0$ in Prop.~\ref{convex4}.
  Then  the first-order condition is
\beqq\label{Ln1}
q_0:={z_0}^{-1}\odot (A_\bot^* A_\bot z_0)\in \IR^N, \eeqq
and the second-order necessary condition is that 
for all $\xi\in \Xi:=\{\xi\in \IR^N; \langle b,\xi\rangle=0\}$,
\beqq\label{Ln2}
\xi^\top (K_{z_0}^\bot-\diag(q_0))\xi\ge 0.
\eeqq
Hence, $z_0$ is  a strictly  local minimizer on $\cZ$, if 
\beqq\label{Ln3}
\|\xi\|^{-2}\xi^\top (K_{z_0}^\bot-\diag(q_0))\xi> 0,\; \xi\in\Xi,\; \|\xi\|>0
\eeqq
and  (\ref{Ln1}) hold. 
 
 \end{cor}
 \begin{proof}
 
This  is the special case of Prop.~\ref{convex4} with $\lambda=0$. Note that  $q_0=q(z,0)$ and
 we have  $q_0\in \IR^N$ from (\ref{eq3.18}) and
  \beqq
\rho=\langle b^2, \rho 1_N\rangle=\langle b^2,\Im( q(z,\lambda))\rangle=\Im(\|A_\bot z\|^2)=0.
\eeqq

 \end{proof}

Readers should notice  
different  tangent spaces  used in  Prop.~\ref{convex1} and Prop.~\ref{convex4}. 
%Actually, we observe that   (\ref{3.10}) does not hold 
%in the early iterations of RAAR, which is consistent to the definition of $\Xi$.  See experiments. 
Finally, we  show that $(z_*, \lambda_*)$ can be  a local saddle   under sufficient small $\beta$, if $z_*$ is a strictly local minimier. For the application of Fourier phase retrieval,    Theorem~\ref{thm3.2} shows the existence of a strictly local minimizer  of (\ref{main_P}) based on the  spectral gap condition  of   coded diffraction patterns.
 
  \begin{prop}\label{2.15}
  Let  $z_*$ be a   strictly local minimizer in (\ref{main_P}). 
 Let  $\lambda_*=-\beta' A_\bot^* A_\bot z_*$. Then we can find 
   $\beta$ satisfying 
\beqq\label{H2c1}  
\|\xi\|^{-2} \left< \xi, (K_{z_*}^\bot  -\diag(b^{-1}\odot K_{z_*}^\bot b))\xi \right>> \beta \|K^\bot \xi\|^2\ge 0 \textrm{ for any $\xi\in\Xi$,}
 \eeqq
such that 
  %$\alpha z_*$ is a local minimizer  for each complex unit $\alpha$. In addition, 
  $( z_*,\lambda_*)$  is a  local saddle of 
  \beqq\label{L''eq''}
\max_{\lambda}\min_{z\in \cZ'(z_0)}  \left\{F(z,  \lambda):=\frac{\beta}{2}\|A_\bot ( z-  \lambda)\|^2-\frac{1}{2}\| \lambda\|^2  \right\}.
\eeqq 

\end{prop}
\begin{proof} 
Since $z_*$ is a strictly local minimizer,
then
   Cor.~\ref{local_cond} gives $q_0:=z_*^{-1}\odot  (A_\bot^* A_\bot z_*)\in \IR^N$ and 
    $\|\xi\|^{-2}\left< \xi, (K^\bot-\diag(q_0)) \xi\right>>0$. 
With $(1+\beta')=(1-\beta)^{-1}$ and  \[ q(z_*, \lambda_*)=z_*^{-1}\odot A_\bot^* A_\bot (z_*- \lambda_*)
 =(1+\beta') z_*^{-1}\odot  (A_\bot^* A_\bot z_*)=(1+\beta') q_0\in \IR^N,\]
we have the second-order sufficient condition (\ref{Ln3}) for $\beta$ satisfying (\ref{H2c1}),
\beqq\label{H2c1'} 
(1-\beta)  \left< \xi, ( K^\bot -\diag(q(z_*, \lambda_*)))\xi\right>=\left< \xi, ( K^\bot -\diag(q_0))\xi\right>-\beta \|K^\bot \xi\|^2.
\eeqq
\end{proof}

Next, we show that with probability one, the global solution  $z=A^*x_0$ is a strictly local minimizer of (\ref{main_P}) in the following Fourier phase retrieval. 
The main tool is 
 the following spectral gap theorem in~\cite{CHEN2018665}.
 
\begin{theorem}[Theorem 6.3\cite{CHEN2018665}]\label{thm6.3}Let $\Phi$ be the oversampled discrete Fourier transform. 
Let $x_0$ be a given rank $\ge 2$ object and at least one of $\mu_j$, $j=1,\ldots, l\ge 2$,  be continuously and independently distributed on the unit circle.
Let 
\beqq\label{eq3.26}
 A^*=c_0\left(
\begin{array}{c}
\Phi\diag\{\mu_1\} \\
\vdots\\
\Phi\diag\{\mu_l\}
   \end{array}
    \right)\eeqq
   be isometric with a proper choice of $c_0$ and $B:= A\diag(u_0)$, $u_0=|A^* x_0|^{-1}\odot (A^* x_0)$. Then with probability one, 
   \beqq
\label{3.24}\lambda_2=\max\{ \|\Im(B^* v)\| : v\in \IC^n, v\bot i x_0, \|v\|=1\}<1.   
   \eeqq

\end{theorem}

\begin{theorem}\label{thm3.2}Let $\Phi$ be the oversampled discrete Fourier transform. 
Let $x_0$ be a given rank $\ge 2$ object and at least one of $\mu_j$, $j=1,\ldots, l\ge 2$,  be continuously and independently distributed on the unit circle.
Let 
\beqq
 A^*=c_0\left(
\begin{array}{c}
\Phi\diag\{\mu_1\} \\
\vdots\\
\Phi\diag\{\mu_l\}
   \end{array}
    \right)\eeqq
   be isometric with a proper choice of $c_0$.  Then with probability one, $z=A^* x_0$ is a strictly local minimizer of (\ref{main_P}), and
   \beqq\label{gap}
\min_{\xi} \{ \|\xi\|^{-2}\langle \xi, (K_{z}^\bot-\diag(\Re(q(z, 0))))\xi\rangle: \xi\in \IR^N, \langle |z|,\xi \rangle=0 \}\\
\ge 
1-\lambda_2>0,
\eeqq where  $\lambda_2$ is given in (\ref{3.24}).
\end{theorem}
\begin{proof}

Note that (\ref{3.24}) implies that $\Im(B^* v)=\lambda_2\xi$ holds for some unit vector  $\xi\in \IR^N$. Then $\xi$ is one left singular vector of 
\beqq \cB:=[\Re(B^*), \Im(B^*)]\in \IR^{N\times 2n}\eeqq with singular value $\lambda_2$, while and the associated right singular vector is $[\Im(v)^\top , \Re(v)^\top]^\top$.
Let $c=K|z|$ and  $z=Ax_0$.  The    left and right singular vectors $\cB$ corresponding to singular value $1$ are
 $c\in \IR^N$ and $[\Im(ix_0)^\top, \Re(ix_0)]^\top $, respectively.
Since $\cB^\top  \cB=\Re(B^* B)=\Re(\diag(\overline{u} )A^*A\diag(u))=K_{z}$, then $\xi, c$ are eigenvectors of $K_{z}$.
From theorem~\ref{thm6.3}, with probability one we have
\beqq\label{3.27}
1>\lambda_2\ge \max_{\xi} \{ \langle\xi, K_{z}\xi \rangle: \|\xi\|=1,\langle |z|,\xi\rangle=0\}.
\eeqq
 By definition in (\ref{q_def}),  $q(z, 0)=b^{-1}\odot (K_{z}^\bot b)=1-b^{-1}\odot K_z b=0$. Finally,  since $K^\bot_{z}=I-K_{z}$,  (\ref{3.27}) gives (\ref{gap}) and  (\ref{eq_40_}).
\end{proof}

 \section{RAAR Convergence analysis}\label{sec:4}

\subsection{Convergence of RAAR}

We  shall derive one inequality stated in~(\ref{key1}), which  ensures the  convergence of RAAR iterations $\left\{w_k: k=1,2,\ldots \right\}$  in Prop.~\ref{main}. 
In Theorem.~\ref{Main1}, we shall show that
 the condition  in (\ref{key1})  holds
  near local saddles under a  sufficient large penalty $1/\beta'$. 

From the $\lambda$-iterations of ADMM in (\ref{la11}), we have
\beqq\label{st1}
\lambda_{k+1}=\lambda_k+(y_{k+1}-z_{k+1}), \textrm{ and }
w_{k}:=
\lambda_k+ y_{k+1}=\lambda_{k+1}+  z_{k+1}.
\eeqq
The $z$ iteration yields
$
z_{k+1}=[w_{k}]_{\cZ}=[ z_{k+1}+\lambda_{k+1}]_{\cZ}
$
and  $z_{k+1}+\lambda_{k+1}$ shares the same phase with $z_{k+1}$. We have a lower bound,
\beqq\label{lambda_lb} \lambda_{k+1}
\odot \frac{\overline{z_{k+1}}}{|z_{k+1}|}\in (-|z_{k+1}|,\infty).
\eeqq
With $y_*=z_*$, $\lambda_*=-\beta'A_\bot^* A_\bot z_*$,  and $y_{k+1}=(I-{\beta}A_\bot^* A_\bot) (z_k-\lambda_k)$, ${\beta'}={\beta}/(1-{\beta})$, introduce 
\begin{eqnarray}&&T(z_k,\lambda_k)
:={\beta'}  \|A_\bot (y_*-y_{k+1})\|^2+\|y_{k+1}-z_k\|^2\nonumber\\
&=& 
{\beta} (1-{\beta}) \|A_\bot ((z_k-z_*)-(\lambda_k-\lambda_*))\|^2\nonumber\\
&&
+\|A_\bot (-{\beta} (z_k-z_*)-(1-{\beta})(\lambda_k-\lambda_*))\|^2
+\|A(-\lambda_k)\|^2 \nonumber \\
&=&{\beta}  \|A_\bot (z_k-z_*)\|^2+(1-{\beta})\|A_\bot (\lambda_k-\lambda_*)\|^2
+\|A\lambda_k\|^2\label{T2}.
\end{eqnarray}
We shall derive a few inequalities  from the optimal condition of $y$ and $z$ in (\ref{Leq}), respectively.

\begin{prop}\label{1.11} 
Let $T$ be defined in (\ref{T2}). Then 
\beqq\label{37}
\|w_{k-1}-w_*\|^2-\|w_{k}-w_*\|^2
\ge 
T(z_k,\lambda_k)
+2\left<z_{k}-z_*,  \lambda_k-\lambda_*\right>
\eeqq
\end{prop}
\begin{proof}
Let $C$ be the cost  function, $C(y)={\beta'} \|A_\bot y\|^2/2$, for $ y\in \IC^N$.
We shall prove
\beqq\label{T3}
\frac{1}{2} \|z_k-z_*\|^2\ge 
\left< \lambda_k-\lambda_*,  y_{k+1}-y_*\right>+\frac{1}{2}\| y_{k+1}-y_*\|^2+\frac{1}{2}T(z_k,\lambda_k).
\eeqq
To this end, 
 we make two 
claims. First,
the optimality of $y_{k+1}$ in (\ref{y1}) indicates that for all $y\in \IC^N$,
\begin{eqnarray}
C(y)-C(y_{k+1})&=&-\frac{1}{2} \|y-z_k\|^2+\left< \lambda_k,  (y_{k+1}-y)\right>+\frac{1}{2} \|y-y_{k+1}\|^2\\
&+&
\left(\frac{1}{2} \|y_{k+1}-z_k\|^2
+\frac{{\beta'}}{2} \|A_\bot (y-y_{k+1})\|^2\right).\label{Q2a}
\end{eqnarray}
Second, the optimality of $y_*$ in $C(y)$ indicates
\beqq\label{Q2b}
C(y)-C(y_*)+\left< (y-y_*),  \lambda_*\right>\ge 0.
\eeqq

To verify  (\ref{Q2a}),  use the optimality
$y_{k+1}=(I+{\beta'} {A_\bot^*} A_\bot )^{-1} (z_k-\lambda_k)$  in (\ref{y1}), which gives
 $ \nabla_y L(y_{k+1}, z_k, \lambda_k)=0$.
 The quadratic convexity  of $L$ in $y$ gives
%\beqq\label{L0}
%L(y, z_k, \lambda_k)= L(y_{k+1}, z_k, \lambda_k)+\left< (y-y_{k+1}), \nabla_y L(y_{k+1}, z_k, \lambda_k)\right>+(y-y_{k+1})\cdot \nabla_{yy} L(y_{k+1}, z_k, \lambda_k) (y-y_{k+1}),
%\eeqq
%which implies
 (\ref{Q2a}), i.e., 
\begin{eqnarray*}
&&L(y, z_k, \lambda_k)=C(y)+\left< \lambda_k,  (y-z_k)\right>+\frac{1}{2} \|y-z_k\|^2\\
&=&
C(y_{k+1})+\left< \lambda_k,  (y_{k+1}-z_k)\right>+\frac{1}{2} \|y_{k+1}-z_k\|^2+ (\frac{{\beta'}}{2} \|A_\bot (y-y_{k+1})\|^2
+\frac{1}{2} \|y-y_{k+1}\|^2
).
\end{eqnarray*}
For (\ref{Q2b}),  with  $\lambda_*=-{\beta'} A_\bot^* A_\bot z_*=-{\beta'} A_\bot^* A_\bot y_*$,
Taylor's expansion of $C(y)$ at $y_*$ gives 
\begin{eqnarray}
&& C( y)-C(y_*)= {\beta'} \left< A_\bot^* A_\bot y_*,  y-y_*\right>+
\frac{{\beta'}}{2} \|A_\bot(y_*-y)\|^2
\ge \left< -\lambda_*,   y-y_*\right>.
\end{eqnarray}
With $y=y_*=z_*$ in  (\ref{Q2a}) and  $y=y_{k+1}$ in (\ref{Q2b}), (\ref{Q2a})+(\ref{Q2b}) gives (\ref{T3}).
%\begin{eqnarray}
%0&\ge & 
%\left< \lambda_k -\lambda_*,  y_{k+1}-z_*\right>
%-\frac{1}{2} \|z_*-z_k\|^2+\frac{1}{2} \|z_*-y_{k+1}\|^2+\frac{1}{2}T(z_k,\lambda_k).
%\end{eqnarray}

Next, from (\ref{st1}), we have two identities,
\begin{eqnarray}\label{sq1}
&&(y_{k+1}-z_*)
=w_{k}-w_*-(\lambda_k-\lambda_*),\\
&&\label{sq2}
(z_{k}-z_*)
=w_{k-1}-w_*-(\lambda_k-\lambda_*).
\end{eqnarray}
The difference of the squares of (\ref{sq1}) and (\ref{sq2}) gives
\begin{eqnarray}
&&
- \|z_*-z_k\|^2+ \|z_*-y_{k+1}\|^2\nonumber \\
&=&  
\|w_{k}-w_*\|^2-\|w_{k-1}-w_*\|^2
-2\left<w_{k}-w_{k-1},  \lambda_k-\lambda_* \right>\label{end:2}
\\
&=&
\|w_{k}-w_*\|^2-\|w_{k-1}-w_*\|^2
+2\left<z_{k}-z_*,  \lambda_k-\lambda_*\right>\nonumber\\
&&-2\left<\lambda_k -\lambda_*,  y_{k+1}-z_*\right>,
\label{end:1}
\end{eqnarray}
where the last equality is given by the difference of (\ref{sq1}) and (\ref{sq2}).  
The proof of (\ref{37}) is completed by  (\ref{end:1}) and (\ref{T3}).
\end{proof}

%\subsubsection{Non-convexity in  $z$ and local convergence}
 Note that for each  fixed point  $w_*:=z_*+\lambda_*$  of RAAR,  $\alpha w_*$ is also a fixed point of RAAR with any complex unit $\alpha$.
\begin{prop}\label{main} For $z,\lambda\in \IC^N$, let $\alpha $ be  the corresponding  global phase factor between $w$ and $w_*$,
\beqq\label{alpha_w}
arg\min_{\alpha }\left\{  \|w-\alpha w_*\|: |\alpha|=1\right\},\; w=z+\lambda, \; w_*:=z_*+\lambda_*.\eeqq
Suppose that there exists  some constant $c_0>1$ such that the following inequality 
\beqq\label{key1}
 T(z,\lambda)\ge  2c_0
\left< \alpha z_*-z, \lambda-\alpha \lambda_*\right>
\eeqq
holds for $(z,\lambda)=(z_k,\lambda_k)$ for $k\ge k_0$. Then  any limit point $(z',\lambda')$ of RAAR satisfies 
\[
A_\bot (z'-\alpha z_*)=0,\; \lambda-\alpha \lambda_*=0 \; \textrm{ for some complex unit $\alpha$}. 
\]
 \end{prop}

\begin{proof} Recall $w_{k-1}=z_k+\lambda_k$ and $w_*=z_*+\lambda_*$. 
Let  $\alpha_k$ be 
some global factor in 
 $ \alpha_k:=arg\min_{|\alpha|=1}\|w_k-\alpha w_*\|.$
Summing ~(\ref{37}) over $k=k_0,\ldots, k_1$, for some global phase factors $\alpha_{k_0},\ldots, \alpha_{k_1}$,
\begin{eqnarray}&& \|w_{k_0-1}-\alpha_{k_0-1}w_* \|^2-\|w_{k_1-1} -\alpha_{k_1-1}w_*\|^2 \\
&=&\|z_{k_0}+\lambda_{k_0}-\alpha_{k_0-1}(z_*+\lambda_*) \|^2-\|z_{k_1}+\lambda_{k_1} -\alpha_{k_1-1} (z_*+\lambda_*)\|^2 \\
&\ge & \sum_{k=k_0}^{k_1-1} \left\{ \|z_{k}+\lambda_{k}-\alpha_{k-1}(z_*+\lambda_*) \|^2-\|z_{k+1}+\lambda_{k+1} -\alpha_{k}(z_*+\lambda_*)\|^2 \right\}\\
& \ge& \sum_{k=k_0}^{k_1-1}
\left\{ T(z_k,\lambda_k)-2\left<\alpha_{k-1} z_*-z_{k}, \lambda_{k}-\alpha_{k-1}\lambda_*\right> \right\}\\
&\ge& (1-c_0^{-1}) \sum_{k=k_0}^{k_1-1} T(z_k,\lambda_k).
\end{eqnarray}
Hence, 
\beqq\label{eq81}
( 1-c_0^{-1}) (k_1-k_0) \left\{\min_{k=k_0,\ldots, k_1-1} T(z_k,\lambda_k)\right\}\le  \|w_{k_0-1}-\alpha_{k_0-1}w_* \|^2.
\eeqq
Let $k_1\to \infty$. Since the left-hand side is bounded above and $1-c_0^{-1}>0$, then 
\[
\liminf_{k\to \infty} T(z_k,\lambda_k)=0.
\]
Let $(z',\lambda')$ be a limiting point and $\alpha$ be the limiting phase factor. For  ${\beta}\in (0,1)$,
 from (\ref{T2})
\beqq\label{51}
A\lambda'=0,  A_\bot(-{\beta} (z'-\alpha z_*)-(1-{\beta}) (\lambda'-\alpha \lambda_*))=0=A_\bot ((z'-\alpha z_*)-(\lambda'-\alpha \lambda_*)).
\eeqq
The second part of (\ref{51}) gives  $A_\bot \lambda'=\alpha A_\bot \lambda_*$. Thus $\lambda'=\alpha \lambda_*$ and   $A_\bot z'=\alpha A_\bot z_*$. 

 \end{proof}

% \begin{rem}
% Note that with the spectral gap condition on $[\Re(A\diag(u_*)), \Im (A\diag(u_*))]$, we actually  have $z'=z_*$ for some global solution $z_*$. 
%\end{rem}
\begin{rem}When (\ref{key1}) holds eventually, then (\ref{eq81}) indicates that 
 $T(z_k,\lambda_k)$
 sub-linearly
  converges to  $0$, i.e.,  (\ref{T2}) indicates  sub-linear  convergence  of RAAR,  $O((k_0-k_1)^{-1})$. This is consistent with   sub-linear convergence in FDR numerical experiments in \cite{CHEN2018665}.  
\end{rem}

 \subsection{Justification of (\ref{key1}) from local saddles}
Next  we shall  verify  that  the  convergence condition in (\ref{key1}) holds,i.e., 
\beqq
{\beta}  \|A_\bot (z-z_*)\|^2+(1-{\beta})\|A_\bot (\lambda-\lambda_*)\|^2
+\|A\lambda\|^2
> 2 \left< z_*-z, \lambda-\lambda_*\right>.
\eeqq
 if 
the positive definite condition  (\ref{m3}) holds at $z_*$ and $(z,\lambda)$ is close  to $(z_*, \lambda_*)$.
For the sake of simplicity, we shall omit the global factors $\alpha$ in front of $z$ and $\lambda$, if no confusion occurs. 
%The non-negativity of $T$  originates   from the convexity of $y$ in Prop.~\ref{1.10}.
\begin{rem}\label{Tbound} 
With  ${\beta}\in (0,1)$,
$T(z,\lambda)$  can quantize  one distance between $(z,\lambda)$ and $(z_*, \lambda_*)$. That is, 
for   $\epsilon>0$,  from (\ref{T2}),  $T(z,\lambda)<\epsilon$ implies
 \beqq\label{bT}
\max\left\{  {\beta}  \|A_\bot (z-z_*)\|^2,(1-{\beta})\|A_\bot (\lambda-\lambda_*)\|^2,
\|A\lambda\|^2\right\}\le  \epsilon.
\eeqq
Thus, $
 \|A_\bot (z-z_*)\|^2\le \epsilon/{\beta},\; \|\lambda-\lambda_*\|^2\le \epsilon/(1-{\beta}).
$
\end{rem}

%
%
%The following gives the $\beta$-condition for the local convergence of $\beta$-RAAR to a local saddle $(z_*, \lambda_*)$.

\begin{theorem} \label{Main1}

Let  $z_*$ be a strictly local minimizer in (\ref{main_P}).
Then we can find  $\beta\in (0,1)$ satisfying
\beqq  
 (1- \beta) \left< \xi, K_{z_*}^\bot\xi \right>  > 2  \left< \xi,  \diag(b^{-1}\odot (K_{z_*}^\bot b))\xi\right> \textrm{ for any unit vector $\xi\in \Xi$.}
\label{m3}\eeqq 
Consider $\beta$-RAAR with this ${\beta}\in (0,1)$. Let
 $
\lambda_*=-{\beta'} A_\bot^* A_\bot z_*$ 
and
 $w_*=z_*+\lambda_*$. Then 
  there is some constant $\epsilon>0$, 
such that  
(\ref{key1})  holds for  all $(z,\lambda)$ with 
$ \|w-w_*\|^2<\epsilon$, where    a proper  complex unit  is applied on $w_*$ according to   (\ref{alpha_w}).

\end{theorem}

\begin{proof} First, the existence of $\beta$ for (\ref{m3}) is ensured by Prop.~\ref{2.15}.
 The RAAR iterations satisfying (\ref{z},\ref{la11}, \ref{st1},\ref{lambda_lb}) indicate
  the decomposition $w=z+\lambda$ with $z\in \cZ$ and $\lambda\odot z^{-1}\in \IR^N$. 
Let $u=z/|z|$, $u_*=z_*/|z_*|$ and   $q_0=b^{-1}\odot K^\bot b$. Then 
$ \bar u\odot \lambda\in \IR^N, \bar u_*\odot \lambda_*\in \IR^N$
and  \begin{equation}
(z_*)^{-1}\odot \lambda_*=-\beta' b^{-1}\odot K^\bot b=-\beta' q_0.
\end{equation}
 Using continuity arguments on (\ref{m3}),  we have  with $u'=(z+ z_*)/|z+ z_*|$, $\xi=(-i)\bar u'\odot ( z_*-z)\in \IR^N$,
 \begin{equation}
  \label{eq86}
(1-\beta)
\left< 
\xi, \bar u'\odot (A_\bot^* A_\bot (u'\odot \xi))
\right>
 >
-(\beta')^{-1}
 \left< \xi, \left(\frac{\lambda_*}{z_*}+\frac{\lambda}{z}\right)\odot \xi\right>
 \end{equation}
 for $(\lambda, z)$   sufficiently close to $( \lambda_*, z_*)$. Observe that as $z\to z_*$, we have 
$
\left< \xi, b\right>=0.
$ Note that $T(z,\lambda)$ has an upper bound $2(1-c_0^{-1})^{-1}\|w_{k_0-1}-w_*\|^2/2$ from (\ref{eq81}). 
 According to Remark~\ref{Tbound}, (\ref{eq86}) holds, if $\|w_{k_0-1}-w_*\|<\epsilon$ holds for some $\epsilon$ sufficiently small. 
With  (\ref{eq86}), algebraic computation gives  (\ref{key1}). Indeed, 
\begin{eqnarray*}
&&2c_0 \left< \lambda-\lambda_*, z_*-z\right>\\
&=&
2c_0 \left< \lambda\odot z^{-1},\bar z\odot ( z_*-z)\right>-2c_0 \left<\lambda_*\odot z_*^{-1},\bar z_* \odot (z_*-z)
\right>\\
&=&  
- c_0 \left<\lambda\odot  z^{-1} + \lambda_*\odot  z_*^{-1}, |  (z-z_*)|^2\right>
\\
&\le &  \label{m2}
\beta  \left< \xi, \bar u'\odot (A_\bot^* A_\bot (u'\odot \xi))\right>=
\beta\|
A_\bot(z-z_*)\|^2\le T(z,\lambda).
 \end{eqnarray*}
 Thus,  $\|w_{k_0}-w_*\|<\epsilon$ gives the  closeness condition for the sequential vector $(z,\lambda)$.  
 \end{proof}

% % 
% \begin{figure}
% \begin{center}
%     \includegraphics[width=0.7\text width]{ plot_circle} 
%          \end{center}
%   \caption{    Nearly orthogonality between $\lambda=w-[w]_{\cZ}$ and $z-z_*=[w]_{\cZ}-[w_*]_{\cZ}$; nearly orthogonality between $\lambda_*=w_*-[w_*]_{\cZ}$ and $z-z_*=[w]_{\cZ}-[w_*]_{\cZ}$.
%    }  \label{expRAAR}
%    \end{figure}

%
%
%\begin{rem}   First, since $z_{k+1}$ is not exactly a minimizer associated with  $\lambda_k$,  the RAAR algorithm is a fixed point algorithm, which  approaches   a saddle point in   a slightly different manner. 
%%(Recall saddle points are the intersection of  $\widetilde \cV$ in (\ref{pre_set}) with one affine space in section 2.3). 
%RAAR iterations directly target  a local saddle $(z_*,\lambda_*)$ satisfying  (\ref{H2c1'})  in the intersection of $\cV'$ and the affine space given in (\ref{fb}), where
%  \beqq\label{pre_set1}
%  \cV':= \cup_{z\in\cZ} \left\{(z, \lambda)\in \cZ'\times  \IC^N: %z^{-1}\odot \lambda\in \IR^N, \;  
%  \xi^\top ((1-\beta) K^\bot+{\beta'}^{-1}\diag(z^{-1}\odot \lambda))\xi> 0, \; \forall \xi\in \Xi  \right\}.
%  \eeqq
%Second, as pointed out in remark~\ref{example_p}, a large penalty parameter is required to ensure the convergence of augmented Lagrangian methods. The following condition stated in (\ref{m3}) is tighter than the local minimizer condition  in (\ref{pre_set1})  or (\ref{H2c1'}) due to the  additional factor $2$  on the right hand side of (\ref{m3}).
%Hence,   not every point  in the intersection of  $\cV'$ and the affine subspace can be a limit of $\beta$-RAAR.
%
%\end{rem}

\section{Numerical experiments}

 \subsection{Gaussian-DRS} The $\beta$-RAAR algorithm is not the only  algorithm,  which can screen out  some undesired local saddles via varying  penalty parameters.  Recently,\cite{fannjiang2020fixed} proposed 
 Gaussian-Douglas-Rachford Splitting 
 to solve phase retrieval via minimizing  a loss function  $\||z|-b\|^2$ subject to $z$ in the range  of $A^*$. Let $x$ be the unknown object.  Let  $1_\cF(y)$ be   the indicator function of the range $\cF$ of $A^*$. Then  $A^* x \in \cF$.
 Similar to RAAR,  the algorithm can  be formulated  as  ADMM with a penalty parameter $\rho>0$  to reach
  a   local max-min point  of the Lagrangian function
 \beqq\label{DRS_150}
 \max_{\lambda}\min_{y,z\in\IC^N} \left\{\frac{1}{2}\| |z|-b\|^2+ \left< \lambda,  z-y\right >+\frac{\rho}{2} \|z-y\|^2+1_\cF(y)\right\}.
 \eeqq
 The ADMM scheme consists of repeating the  following three updates to reach a fixed point $(y,z,\lambda)$:
 \begin{itemize}
 \item $y\leftarrow A^*A (z+\rho^{-1}\lambda)$;
 \item 
 $ z\leftarrow (1+\rho)^{-1}(\frac{w}{|w|}\odot b+\rho w)$ where $w=y-\rho^{-1}\lambda$;
 \item $\lambda \leftarrow \lambda+\rho (z-y)$.
 \end{itemize}
{  Similar to the RAAR reconstruction  in (\ref{eq_x}),  once a fixed point of this ADMM is obtained, the object $x$ can be computed by  $x=Ay=A(z+\rho^{-1}\lambda)$ from the $y$-update. } 
Introduce $P=A^*A$ and $P^\bot=I-P$. After eliminating $y$,
%
%  the  ADMM scheme  consists of repeating the following  updates,
% \begin{itemize}
% \item $z\leftarrow  (1+\rho)^{-1} \left\{ [(P z-P^\bot \lambda/\rho)]_\cZ +\rho  (P z-P^\bot \lambda/\rho)\right\}
% $;
% \item $\lambda/\rho\leftarrow\lambda/\rho+   (P^\bot z-P\lambda/\rho)= P^\bot (z+\lambda/\rho)$;
% \end{itemize}
%Likewise, (\ref{DRS_150}) reduces to 
 the  local max-min problem reduces to 
\beqq\label{DRS_151}
\max_\lambda \min_z \left\{\frac{1}{2}\||z|-b\|^2+\frac{\rho}{2} \left\{\|P^\bot (z+\frac{\lambda}{\rho})\|^2-\|\frac{\lambda}{\rho}\|^2\right\}\right\}.
\eeqq
 
 Algebraic computations on ADMM scheme yield   the fixed-point condition of DRS.
 \begin{prop} Denote  $\mu:=\lambda/\rho$.
 Let $(z,\mu)$ be a fixed  point of $\rho$-DRS. Then \beqq\label{cond_mag} z+\rho \mu=[z-\mu]_\cZ=[z]_\cZ,\; 
 P\mu=0 \textrm{  and }P^\bot z=0. 
 \eeqq
 \end{prop}
We skip the proof of  (\ref{cond_mag}).  Note that  the condition  implies 
that  the vector  $z$ shares the same phase vector  $u:=z/|z|$ with $z-\mu$, and
 $[z]_\cZ$ has the $(P,P^\bot)$-decomposition $[z]_\cZ=z+\rho \mu$,
 where $P\mu=0$ and $P^\bot z=0$. 
  Hence,  $\mu/z\in \IR^N$ is a real vectors with  bounds,  $-\rho^{-1} \le z^{-1}\odot \mu\le 1$.

  Next we  derive conditions for  local saddles of $L$ in (\ref{DRS_151}). 
 \begin{prop}\label{saddleDRS}
 Let $(z,\mu)$ be a local saddle of $L$ in  (\ref{DRS_151}). Then  
 the first-order optimal condition is
 \beqq\label{eq_154_1}
  z+\rho \mu=[z]_\cZ,\; P\mu=0, \; P^\bot z=0.
 \eeqq 
 When $\rho>0$,  the concavity of $L$ in $\lambda$ is obvious. Let $K:=\Re(\diag(\bar u) P \diag(u))$ and $ u=z/|z|.$
 The second-order necessary  condition of $z$ in  (\ref{DRS_151}) is  
 \beqq\label{eq_154_2}
 (\rho+1)I-\frac{b}{|z|}\succeq \rho K.
 \eeqq

 \end{prop}
  \begin{proof}
The optimality of $\mu$ in (\ref{DRS_151}) is 
 \beqq \label{eq_153}P^\bot (z+\mu)=\mu, \; \textrm{ which implies }  \; 
 P^\bot z=0 \textrm{ and } P\mu=0.
 \eeqq 
From the  derivative  of $L$ with respect to $z$, 
 the $z$-optimality is
 \[
 \frac{z}{|z|}\odot (|z|-b)+\rho P^\bot (z+\mu)=0, 
 \textrm{ i.e., } 
 z-[z]_\cZ+\rho P^\bot (z+ \mu)=0.
 \]
 Together with (\ref{eq_153}), we have
 $
 z+\rho \mu=[z]_\cZ.
 $
Next,  we derive the second-order necessary condition of $z$.
%Since  $z$ shares the same phase vector with $z-\mu$,\[
%z+\frac{\lambda}{\rho}=z+ \mu=[z]=[z-\mu],\]
Consider a perturbation $z\to z+\epsilon$  with $\epsilon\in \IC^N$. From
 $
|z+\epsilon|=|z| \left(1+\Re(\frac{\epsilon}{z}) +\frac{1}{2}\Im(\frac{\epsilon}{z})^2\right)+o(\epsilon^2),
$
and
\begin{eqnarray}
&&\||z+\epsilon|-b\|^2-\||z|-b\|^2\\
%&=&2\left< z,  \epsilon\right> +\|\epsilon\|^2 -2\left< b,  
%|z| \left(\Re(\frac{\epsilon}{z}) +\frac{1}{2}\Im(\frac{\epsilon}{z})^2\right)\right> +o(\|\epsilon\|^2),
%\\
&=&2\left< z-b\odot \frac{z}{|z|},   \epsilon\right> +\|\epsilon\|^2 -\left< \frac{b}{|z|},  \Im(\epsilon\odot \bar u)^2\right>+o(\|\epsilon\|^2),\label{eq155}
\end{eqnarray}
we have  the second-order condition of $L$,
\beqq\label{eq_159}
\rho \left< \epsilon, P^\bot\epsilon \right> +\|\epsilon\|^2  -\left< \frac{b}{|z|},  \Im(\epsilon\odot \bar u)^2\right> \ge  0.
\eeqq
Taking    $\epsilon=ic\odot u$ for  $c\in\IR^N$ yields
  \beqq\label{DRS_c}
 \|c\|^2\ge \frac{1}{\rho+1}  \left<  \frac{b}{|z|},  c^2\right>+\frac{\rho}{\rho+1} \left<c,  Kc\right>.
 \eeqq

%\end{itemize}
\end{proof}

 Next, we show that a local saddle $(z,\mu)$ is always 
 a fixed point of DRS. 
% There is one big difference between RAAR and DRS. Since the magnitude constraint $|z|=b$ is not strictly enforced on $z$, then  
%at a local saddle $(z, \mu=\rho^{-1}\lambda)$, we have
%\[
%L(z,\lambda)=\frac{1}{2}\||z|-b\|^2.
%\]
%Hence, $(\alpha z, \mu)$ is also a local saddle of $L$ for any complex unit $\alpha$. 
\begin{prop} 
If $(\mu, z)$ is a local max-min point in (\ref{DRS_151}), then $z$ is a fixed point of DRS. \end{prop}
\begin{proof}
At each stationary  point  $z$ of DRS, we have $|z|=Kb$ and  $[z]=b\odot u$ of DRS.
 Taking $c=e_i$ in (\ref{DRS_c}) yields
 \[ (\rho+1) Kb=(\rho+1)|z|\ge b, \; i.e., \;
 \bar u\odot (z-\mu)=Kb-\mu\odot \bar u\ge 0, \] which implies  $[z-\mu]_\cZ=[z]_\cZ$. Together with the $(P, P^\bot)$-decomposition, $[z]_\cZ=z+\rho \mu$, we have the fixed point condition.
\end{proof}

From the second-order condition in (\ref{eq_154_2}), we expect  that DRS with  smaller $\rho$  yields  a  stronger  screening-out ability. 
The next remark illustrates  the screening-out similarity  between RAAR and DRS in the case  $\rho$ close to $0$.
\begin{rem}\label{4.13}
 Roughly, for $\rho$ close to $0$, a local saddle  at some phase vector $u$ of $\beta$-RAAR  would be   a local saddle at the same phase vector  $u$ of $\rho$-DRS, if $\rho$ and $\beta$ satisfy   $\beta^{-1}=\rho+1$.
Indeed,  the second-order necessary condition for RAAR function in (\ref{Ln3}) is given by 
\beqq
K^\bot-\diag(\Re(q))
= -\frac{\beta}{1-\beta}I-K+\frac{1}{1-\beta}(\frac{Kb}{b})\succeq 0.
\eeqq
That is, for any nonzero  $\xi\in\IR^N$,
\beqq\label{RAAR_c}
\left< \frac{Kb}{b}, \xi^2\right>\ge \beta\|\xi\|^2+(1-\beta) \xi^\top K\xi.
\eeqq
On the other hand, for DRS, 
 replacing $c$ with 
$\pm (Kb/b)^{1/2}\odot \xi$ and $|z|=Kb$ in (\ref{DRS_c}) gives
\beqq
\left< \frac{Kb}{b}, \xi^2\right>\ge \frac{1}{\rho+1}\|\xi\|^2+\frac{\rho}{\rho+1} \left< (\frac{Kb}{b})^{1/2}\odot  \xi,  K ((\frac{Kb}{b})^{1/2}\odot \xi)\right>.
\eeqq
Comparing with  (\ref{RAAR_c}), 
(\ref{DRS_c}) is almost identical to (\ref{RAAR_c}) under
 $\beta^{-1}=\rho+1$, if
 $\rho$ is close to $0$ and we   ignore the difference of the second terms of their right hand side.   \end{rem}

 \subsubsection{Simulation of Gaussian matrices}
We provide one simulation to present  the screening-out  effect   for $\beta$-RAAR and  $\rho$-DRS.
Generate  Gaussian matrices $A$ with size $n\times N$, $n=100$, $N/n=3, 3.5, 4, 4.5$ and $ 5$, respectively. For simplicity, generate  noise-free data $b=|A^* x_0|$ from some $x_0$. Apply  $\beta$-RAAR and $\rho$-DRS to reconstruct  phase retrieval  solutions under a set of parameters $\beta, \rho$, respectively.  Here, we test  \beqq
 \rho= 1,1/2, 1/3, 1/4, \ldots, 1/10 \textrm{  and  } \beta=1/2, 2/3, 3/4, \ldots, 10/11.\eeqq
  Figures~\ref{expRAAR} show  the success rate of  reaching a global solution for each parameter value (among $40$ trials with   random initializations).  
  %The stagnation only occurs at fixed points of RAAR or DRS, but the stagnation points are not necessarily local saddles. 
  For $\rho$ close to $0$ or $\beta$ close to $1$,  $\beta$-RAAR with   \[
 \beta=\frac{\rho^{-1}}{1+\rho^{-1}}
 \]
gives a similar empirical  performance as $\rho$-DRS, although RAAR performs slightly better.  For instance,  as $\beta\ge 0.8$  or $\rho\le 1/4$,
with   success rates higher than $70\%$,
 $\beta$-RAAR and $\rho$-DRS algorithms both can reconstruct  a global solution in the case  with $N/n\ge 4$. 
These empirical  results are consistent with the theoretical analysis in Remark~\ref{4.13}.
% On the other hand,  for $\beta$  close to $0$,  $\beta$-RAAR gives a similar empirical  performance as $\rho$-DRS, with   \[
% \beta=\frac{1+\rho^{-1}}{2+\rho^{-1}}.
% \]
%
 \begin{figure}
 \begin{center}
     \includegraphics[width=0.45\textwidth]{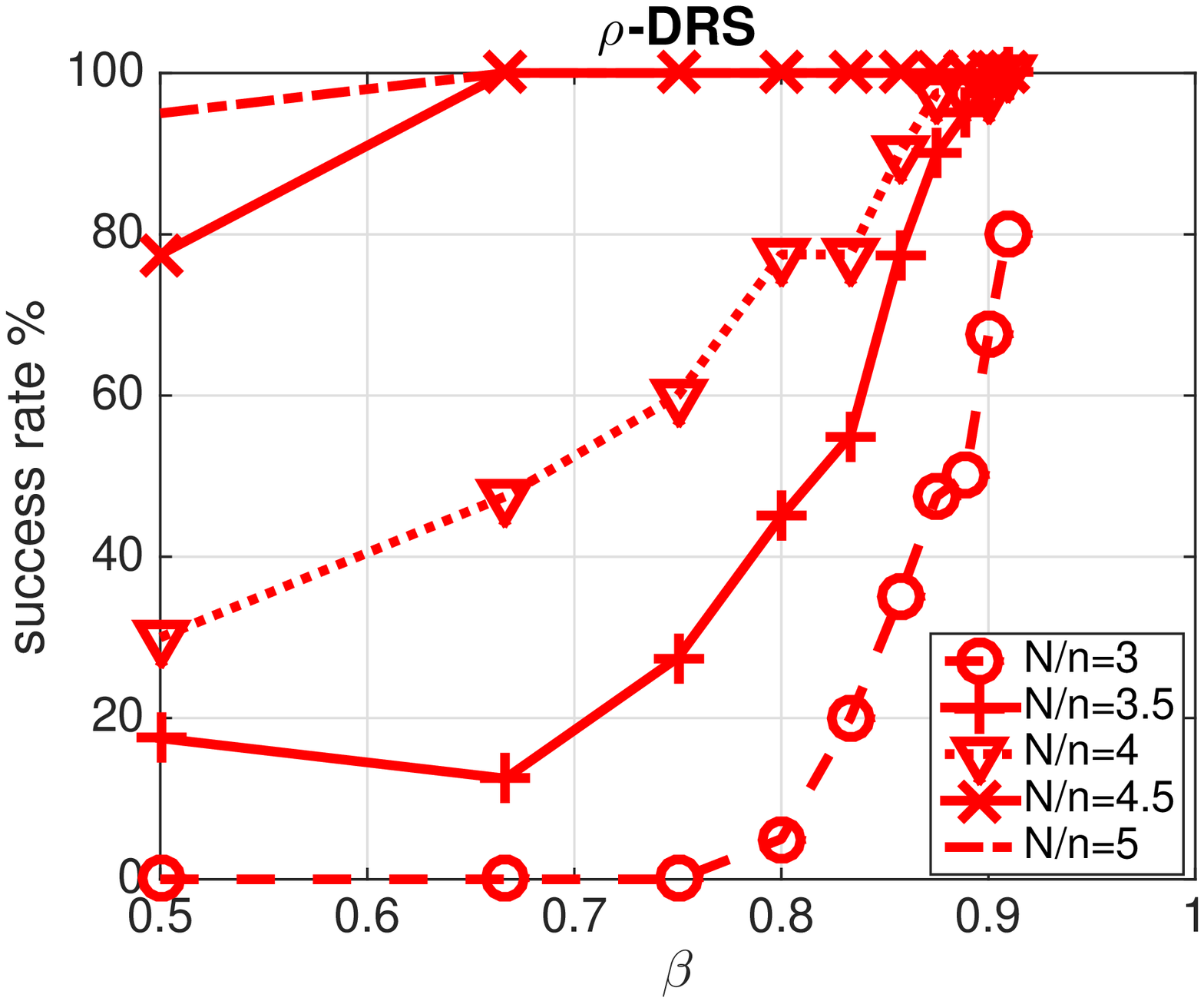}   
        \includegraphics[width=0.45\textwidth]{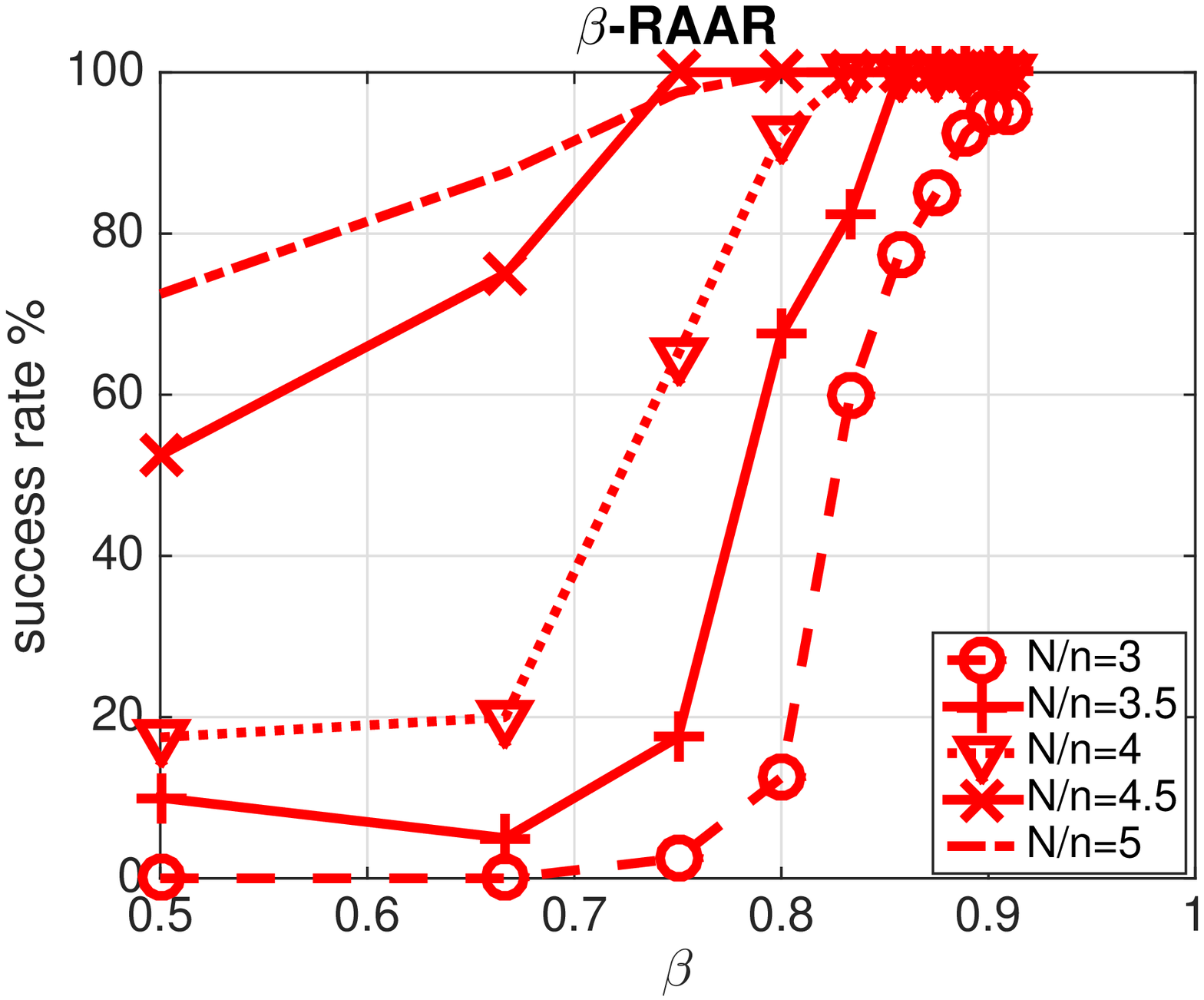} 
          \end{center}
   \caption{   Right and left subfigures show the success rate under $\beta$-RAAR algorithm and  $\rho$-DRS algorithm with $ \beta=\rho^{-1}(1+\rho^{-1})^{-1}$, respectively. }  \label{expRAAR}
    \end{figure}

{ 
\subsection{ Coded diffraction patterns}

The following experiments
present convergence behavior of RAAR on coded diffraction patterns.
 Consider
$1\frac{1}{2}$ coded diffraction patterns with oversampling, i.e., one coded pattern and one uncoded pattern as used in~\cite{CHEN2018665}. For test images $x_0$,  we use  the Randomly Phased Phantom(RPP)
 $x_0=p\odot \mu_0$, where $\mu_0:=e^{i\phi}$ and $\phi$ are i.i.d. uniform random variables over $[0,2\pi]$.  The size  is $128\times 128$, including the margins.
Here, 
 we randomize the original phantom $p$ (in the left of Fig.~\ref{phantom}) in order to make its reconstruction more challenging.  A random object such as RPP is more difficult to recover than a deterministic object.

 \begin{figure}
 \begin{center}
     \includegraphics[width=0.3\textwidth]{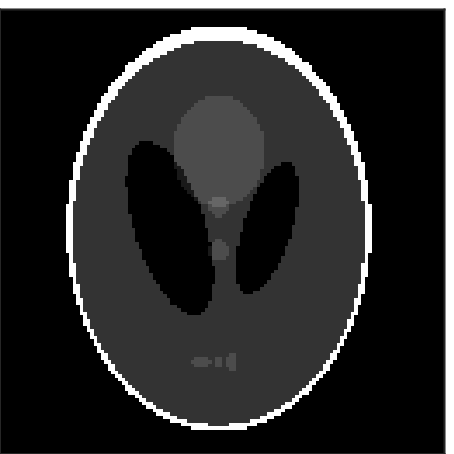}   
     \includegraphics[width=0.3\textwidth]{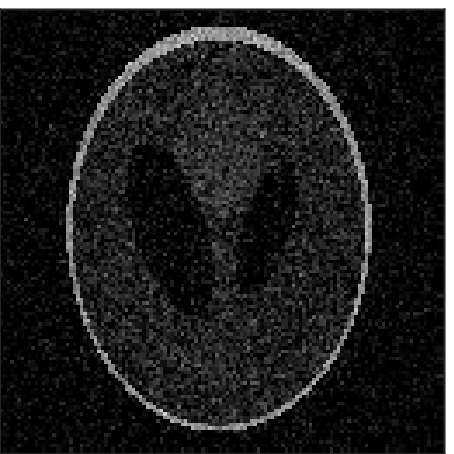}   
     \includegraphics[width=0.3\textwidth]{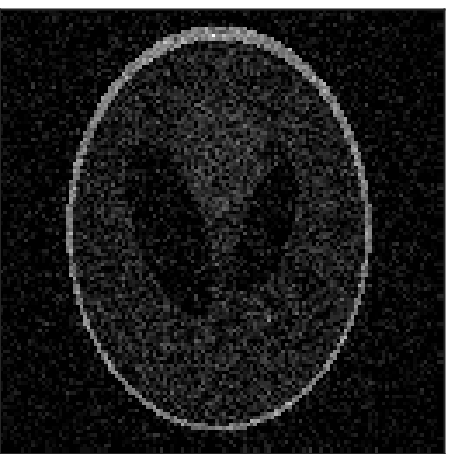}   
          \end{center}
   \caption{   Phantom image $p$ without phase randomization (left);  { Images reprsents the null vector initialization $\Re(x_{null}\odot \overline{\mu_0})$  in the noiseless case(middle) and in the noisy case(right).}}  \label{phantom}
    \end{figure}

Theorem~\ref{thm3.2} states the existence and strictly local minimizer and 
Theorem ~\ref{Main1} indicates that the existence of 
 a local saddle replies on a sufficiently large penalty parameter. 
 The following experiment validates
  RAAR convergence  to the local saddle under proper selection on $\beta$.

Empirically,   $\beta$-RAAR with  large $\beta$ can easily diverge, but   $\beta$-RAAR with  small $\beta$ can easily get stuck (not necessarily converged) near distinct  critical points on $\cZ$. 
To  demonstrate the effectiveness of RAAR, 
 we shall make two adjustments on application of  $\beta$-RAAR. First,  to alleviate the stagnation  at far critical solutions,  we  employ  the  null vector method\cite{Chen2017a}, which is one spectral initialization,  to generate an initialization of RAAR. See the middle and right subfigures in Fig.~\ref{phantom} for the initialization. Second, 
 to reach a local saddle within 600 RAAR iterations, we vary the parameter  $\beta$ along a $\beta$-path, starting  from some initial value  and then decreases to $0.5$, shown in
Fig.~\ref{expRAARbeta}. Each path consists of to phases: (i) $\beta$ remains one constant value selecting from $0.95, 0.9, 0.8, 0.7$ and $ 0.6$ within the first $300$ iterations; (ii) $\beta$  decreases to $0.5$ piecewise linearly within the second $300$ iterations. 
  
 \begin{figure}[htp]
 \begin{center}
                         \includegraphics[width=0.5\textwidth]{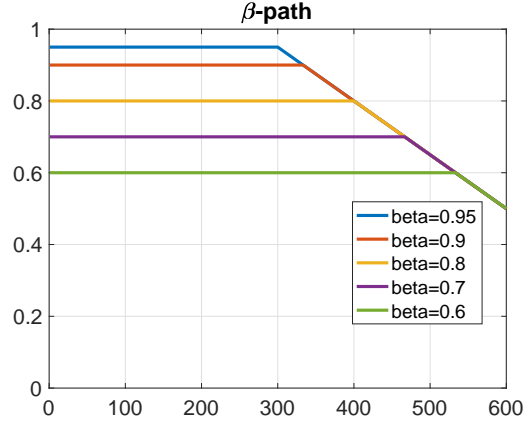}  
                     \end{center}
   \caption{ The corresponding $\beta$ value used within  $600$ RAAR iterations of five different $\beta$-paths. }  \label{expRAARbeta}
    \end{figure}

Conduct  four   experiments to examine $\beta$-RAAR along five $\beta$-paths: 
\begin{enumerate}
\item[(a)] Noiseless data, $b=|A^* x_0|$ with $A$ defined in (\ref{eq3.26}).
Use  the null vector method $x_{null}$ as one initial vector for $\beta$-RAAR, i.e., 
\beqq\label{ini}
z_{1}=[ A^* x_{null} ]_\cZ,\; \lambda_{1}=A^* x_{null}-z_{1}.
\eeqq
\item[(b)]  Noiseless data. RAAR with random initialization. 
\item[(c)]  Noisy data. RAAR with null vector initialization as in (\ref{ini}). 
\item[(d)] Noisy data. RAAR with random initialization. 
\end{enumerate} 
In (c) and (d), the source of noise is the counting statistics~\cite{Thibault2012}, i.e., each entry of the squared measurement $b^2$ follows a Poisson distribution, 
{ 
\beqq
b^2\sim Poisson(\kappa |A^* x_0|^2),\; \kappa>0. 
\eeqq  In  the  RPP experiment,  the parameter $\kappa$ is chosen so that the   noise level   is $\| b-|A^* x_0|\|/\|b\|\approx 0.18$.
}  

\subsubsection{Performance metrics} Results
in the case (a,b) and (c,d)
 are reported in figure~\ref{expRAARcurves2} and  figure~\ref{expRAARcurves2new}, respectively.
Each row 
shows  the performance metrics of 
$\beta$-RAAR iterations $\{w_k\}$. Here, 
$z,\lambda$ are computed from $w$ according to   \beqq \{z_{k+1}:=[w_k]_\cZ, \lambda_{k+1}:=w_k-z_{k+1}\}.
\eeqq
From (\ref{eq_x}), the reconstruction of the object $x$ is estimated by $A(z_{k+1}-\lambda_{k+1})$.
\begin{itemize}
\item Residual: $\|A_\bot z\|/\|b\|$.
% Residual metric tends to $0$, if the iteration reach   the global solution.  
%\item  Scaled objective function $\IF$. Since  $\lambda_*=-\beta' A^*_\bot A_\bot z_*$, 
%\beqq
%F(z_*, \lambda_*)=-\frac{\beta'}{2}\|A_\bot z_* \|^2.
%\eeqq
%Introduce  the  following objective  function $\IF$, so that RAAR iterations  report identical  metric at the same  critical point $z_*$, 
%\[
%\IF(z_k,\lambda_k):=-\frac{1}{\beta'}F(z_k,\lambda_k).
%\]
\item Norm of derivative: 
The Wirtinger derivative  of the objective $F$ in the $\lambda$-direction is 
\beqq
-\{\beta A_\bot^*(A_\bot (z-\lambda)) +\lambda\}.
\eeqq
When   RAAR converges to one local saddle, the derivative norm would be $0$.  
 
The   norm  
\[
\ID_\lambda(z,\lambda):=\left (\|A_\bot ((1-\beta) \lambda+\beta z) \|^2+\|A\lambda\|^2\right )^{1/2}
\]
can be employed
to examine the quality of  convergence. (Empirically, the derivative in the $z$-direction 
 $\ID_z$ has a behavior similar to the one of $\ID_\lambda$. For simplicity, we do not report  $\ID_z$ here.)
\item Inequality ratio $\IT(z_k,\lambda_k)$.
In the noiseless setting, 
 $\IT(z,\lambda)$ is  positive, as $(z,\lambda)$ approaches a local saddle $(z_*,\lambda_*)$ with $\|A_\bot z_*\|=0$. Hence, 
 a positive ratio \[
 \IT(z,\lambda):=1+(\beta \|A_\bot z\|^2+(1-\beta) \|A_\bot \lambda\|^2+\|A\lambda\|^2)^{-1}(2\langle z, \lambda\rangle)
 \] can be used as one indicator that  
 RAAR iterates enter the attraction basin of $(z_*,\lambda_*)$. 
%Another interesting observation is that  the numerical results indicate  $\IT(z,\lambda)\le 1$, which indicates $\langle z, \lambda\rangle\le 0$.
 \end{itemize}
  
\subsubsection{Results on noiseless measurement} In Fig.~\ref{expRAARcurves2}, 
the left column and the right column show  the metric performance in the cases (a) and (b). 

 %The derivative metric enables us to distinguish  different  stagnation:  convergence to other local saddles and divergence. 
 \begin{itemize}
 \item The left column shows the result  of the 600 RAAR iterations along five $\beta$ curves with  null initialization, i.e., case (a). The null initialization is illustrated in 
 the middle of Fig.~\ref{phantom}. Based on the residual and  derivative metrics,  RAAR converges to the global solutions for
 all five  $\beta$-paths. The $\IT$ become positive after $100$ iterations, which  indicates the closeness of the null initialization  to the attraction basin.   In particular,   $\IT$ reaches   $1$ in the early iterations of the case   $\beta=0.95$. 
 
\item The right column shows  the case (b). The initialization difference between case  (a) and case (b) reflects the influence of undesired local saddles. 
 We observe two  distinct convergence behaviors.
  First,  in the case of $\beta=0.6, \beta=0.7$ and $\beta=0.8$,
 based on the metrics of the derivative  norm and the residual, RAAR fails to converge within the first $300$ iterations. As $\beta$ decreases in the second 300 iterations, the iterates tend to different local saddles. 
  Second, for the $\beta$ paths starting with    $\beta=0.9$ or $\beta=0.95$,    RAAR  successfully converge to global solutions. Their  $\IT$-values  are negative in the early 50 iterations, but quickly turn to be  positive after 100 iterations.   Fig.~\ref{expRAARcurves40} demonstrates the reconstruction.

\end{itemize}
\subsubsection{Results on noisy measurements}
The left column  in Fig.~\ref{expRAARcurves2new} demonstrates
 the metric performance of RAAR in  the noisy case (c).
 The null initialization 
 shown in the right of Fig.~\ref{phantom}, is used to reduce the chance of getting stuck at  far local saddles. 
 The  reconstructed objects after the first 300 RAAR iterations are shown in   the top row of  in Fig.~\ref{expRAARcurves4}. Even though these reconstructions  are very similar to the RPP, 
the metric $\ID_\lambda$ indicates that  these RAAR with $\beta\ge 0.7$  fail to converge within the first 300 iterations. Hence, 
we decrease $\beta$ in the second 300 iterations to obtain  local saddles. Observe that  the derivative norm in all cases  decays  to $0$. Actually, by examining  the correlation of reconstructed objects after 600 RAAR iterations,  we verify  that  these five reconstructed objects are identical up to a phase factor. 
   
   The right column  in Fig.~\ref{expRAARcurves2new}  shows   the metric performance of  the noisy case (d). Five $\beta$-RAAR tend to different residual values in the first 300 iterations. 
For large $\beta$, i.e.,  $\beta=0.95$, $\beta=0.9$ and $\beta=0.8$, 
 RAAR  produce 
 rather successful   reconstructed objects shown    the bottom row of  in Fig.~\ref{expRAARcurves4}. These RAAR do not converge within the first 300 iterations. Hence, we reduce the $\beta$ value during the second $300$ iterations.
  By examining  the correlation of reconstructed objects,  we verify that   three final reconstructions are all identical to the final reconstruction in (c). 
  For small $\beta$, RAAR could get stuck at poor solutions, e.g., 
 $\beta=0.7$ and $\beta=0.6$.
Indeed,    after the second 300 iterations, these two  RAAR converge to   non-global local solutions with larger residual values.
The above experiment results suggest that RAAR starting with  large $\beta$ typically performs better than RAAR starting with small $\beta$ in the lack of  spectral methods. In numerical simulations, we demonstrate the effectiveness of RAAR on coded diffraction patterns, where $\beta$ travels from  a large value  to $0.5$.

 \begin{figure}[htp]
 \begin{center}
                        \includegraphics[width=0.45\textwidth]{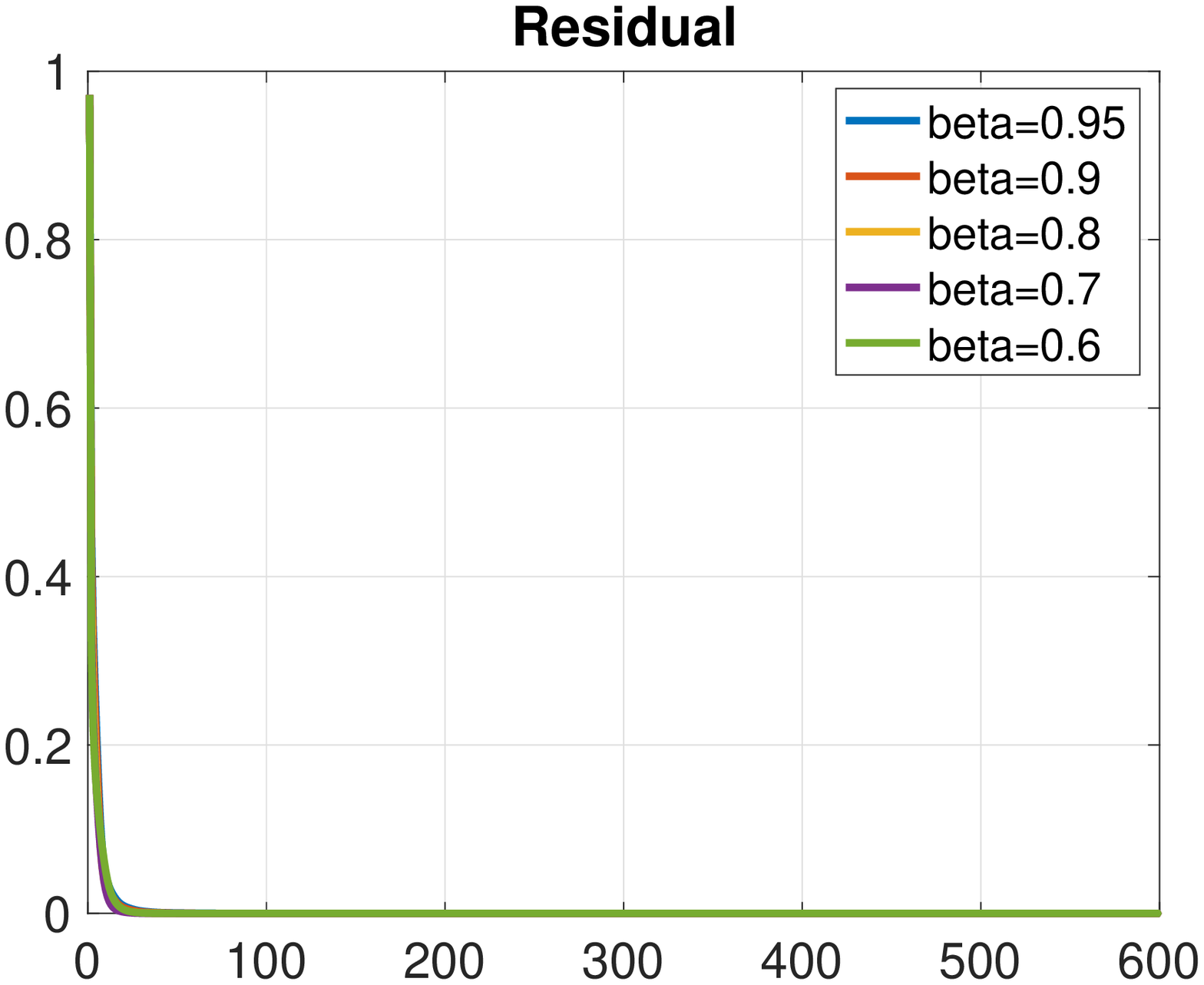}  
            \includegraphics[width=0.45\textwidth]{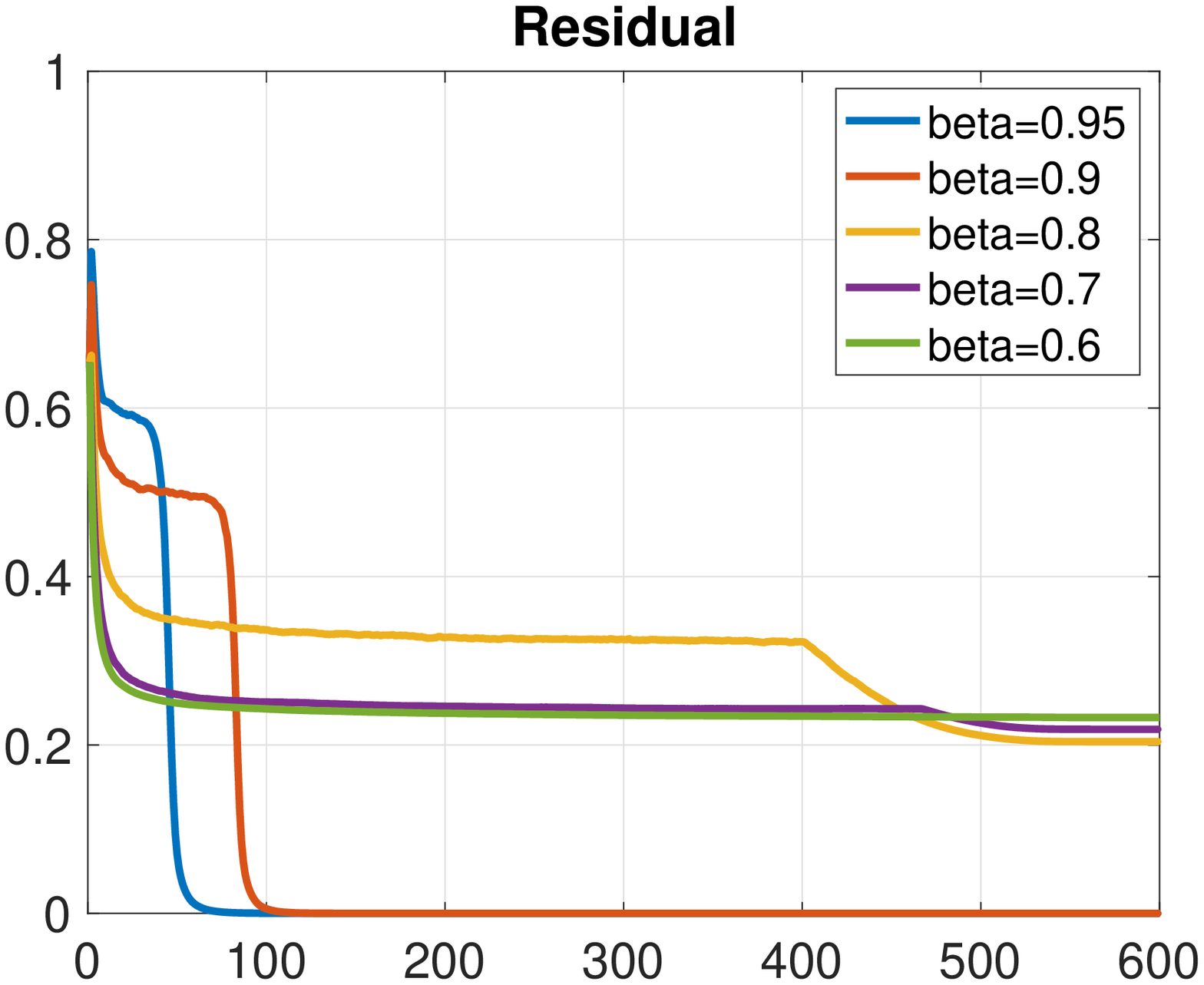} \\
            \includegraphics[width=0.45\textwidth]{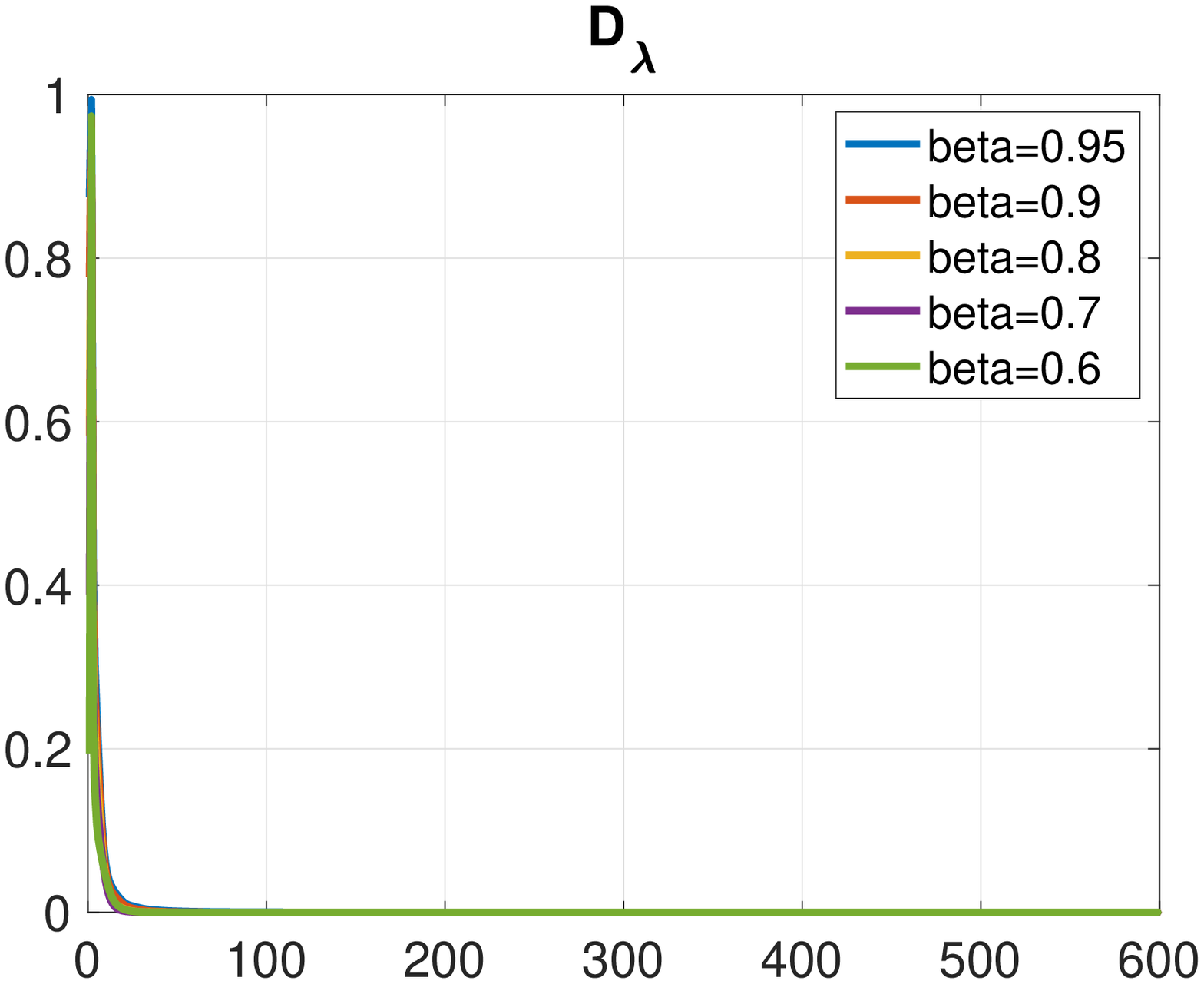}  
            \includegraphics[width=0.45\textwidth]{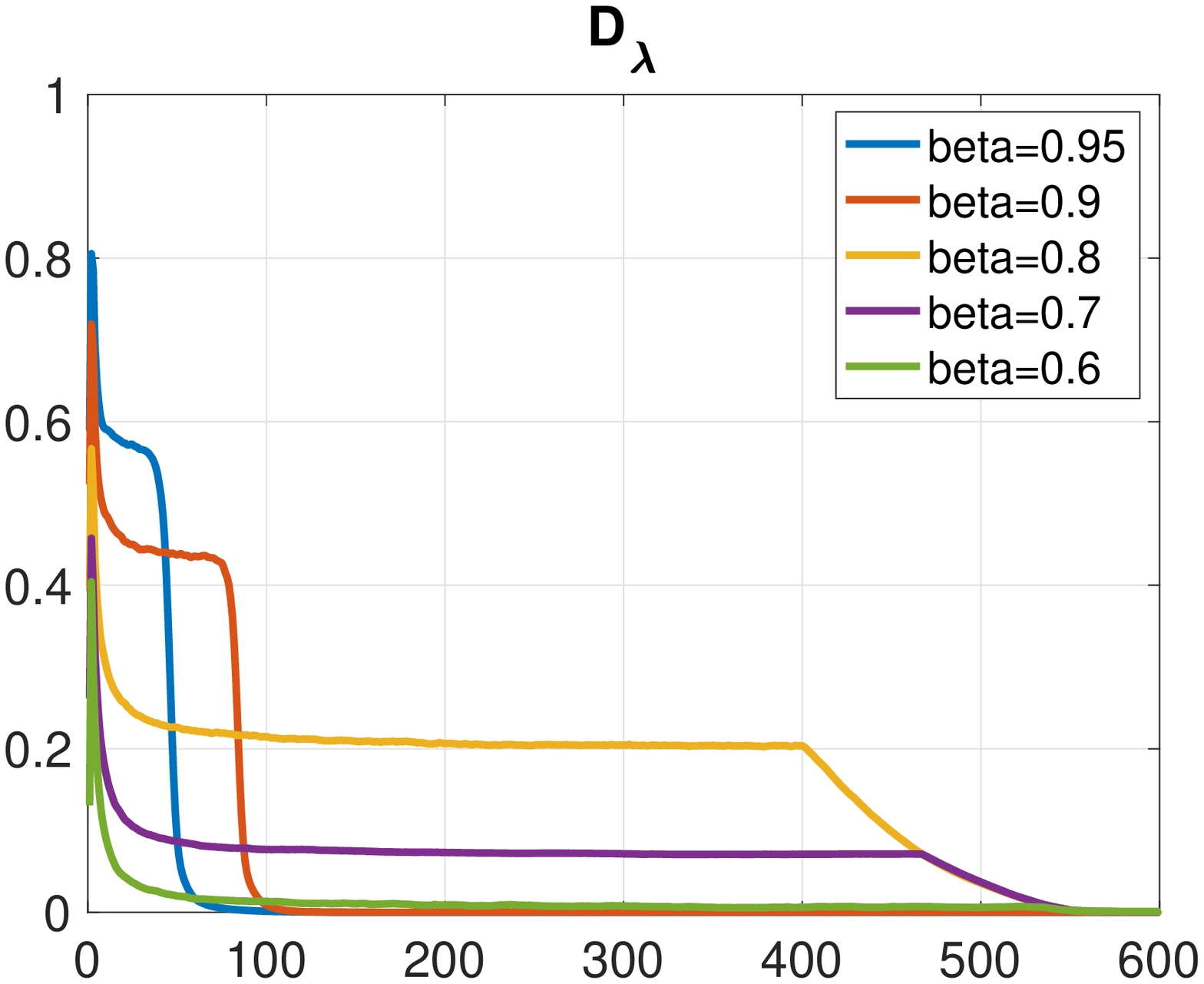}\\
   \includegraphics[width=0.45\textwidth]{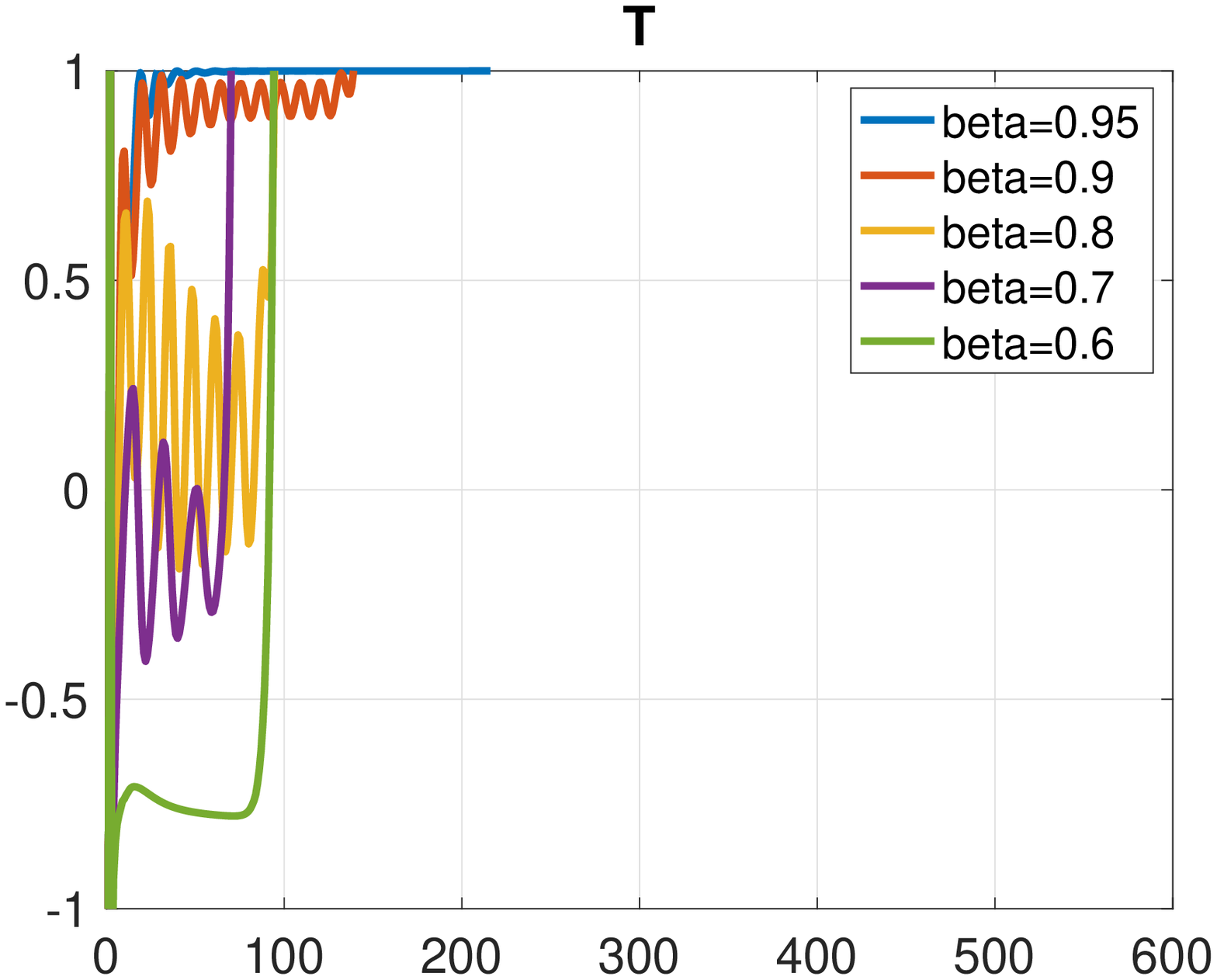}  
            \includegraphics[width=0.45\textwidth]{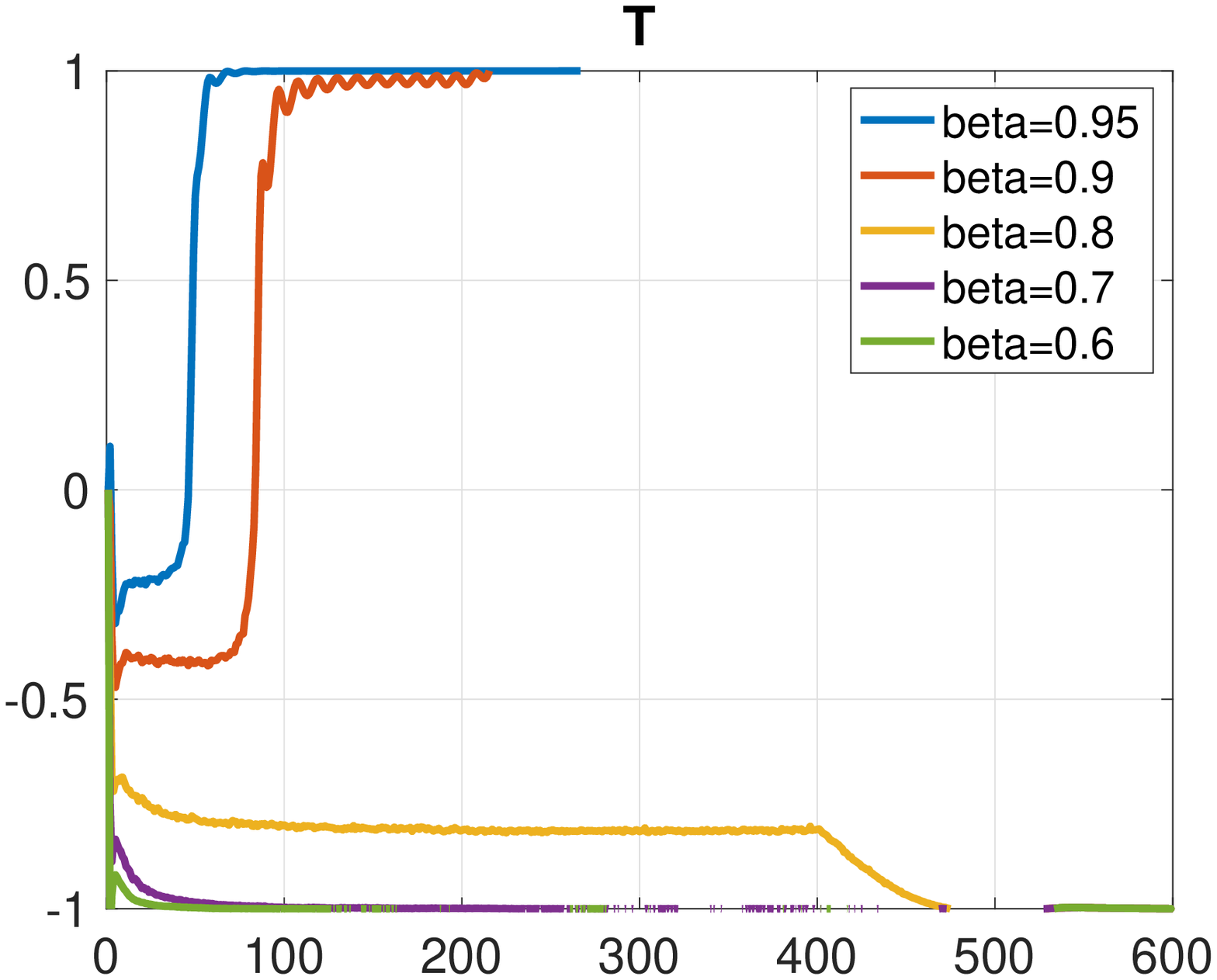}
                      \end{center}
   \caption{ Left and right columns show the residue metric and the derivative metric of RAAR in    
    case(a) and case (b), respectively.  }  \label{expRAARcurves2}
    \end{figure}

 \begin{figure}[htp]
 \begin{center}
            \includegraphics[width=0.45\textwidth]{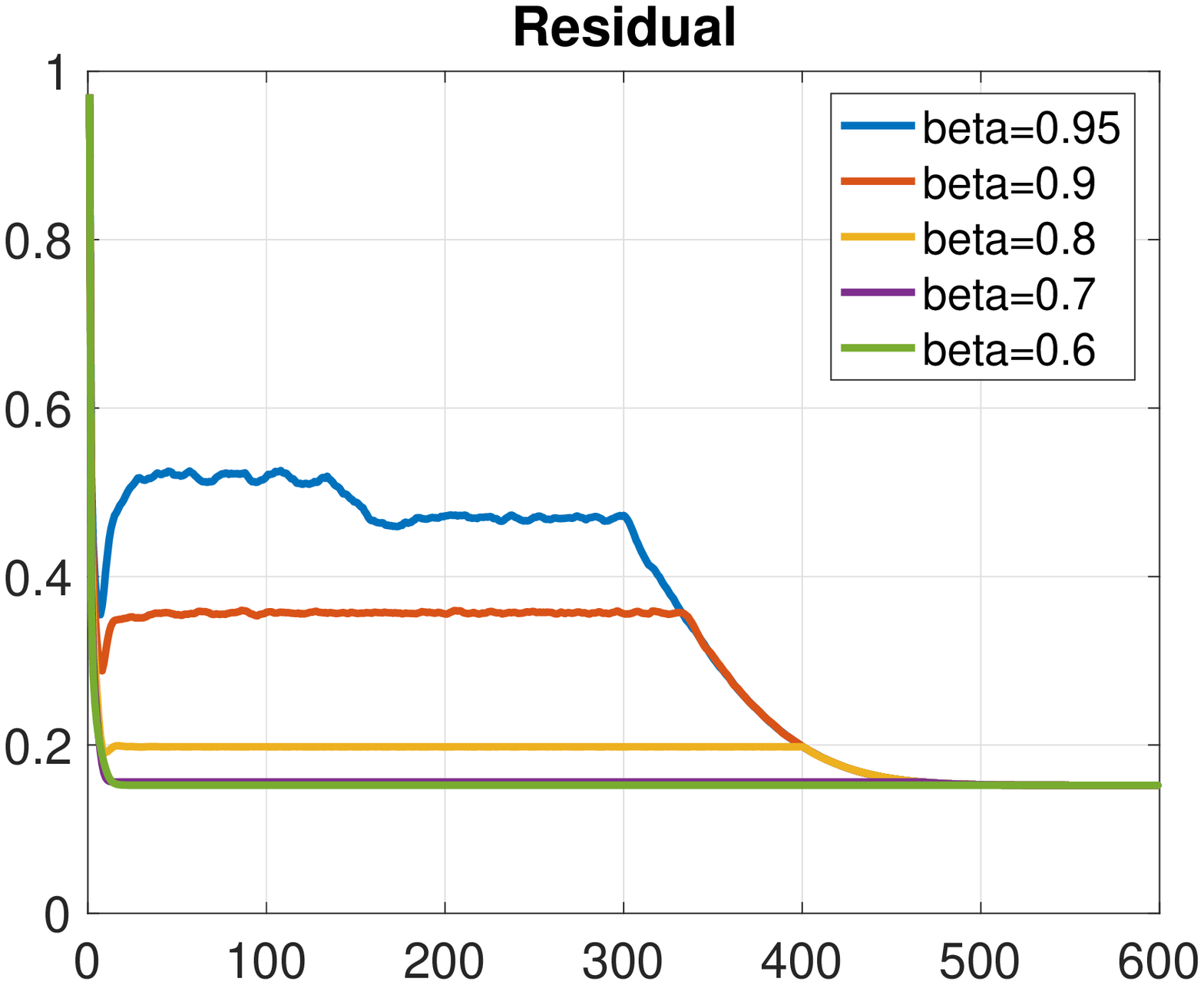} 
             \includegraphics[width=0.45\textwidth]{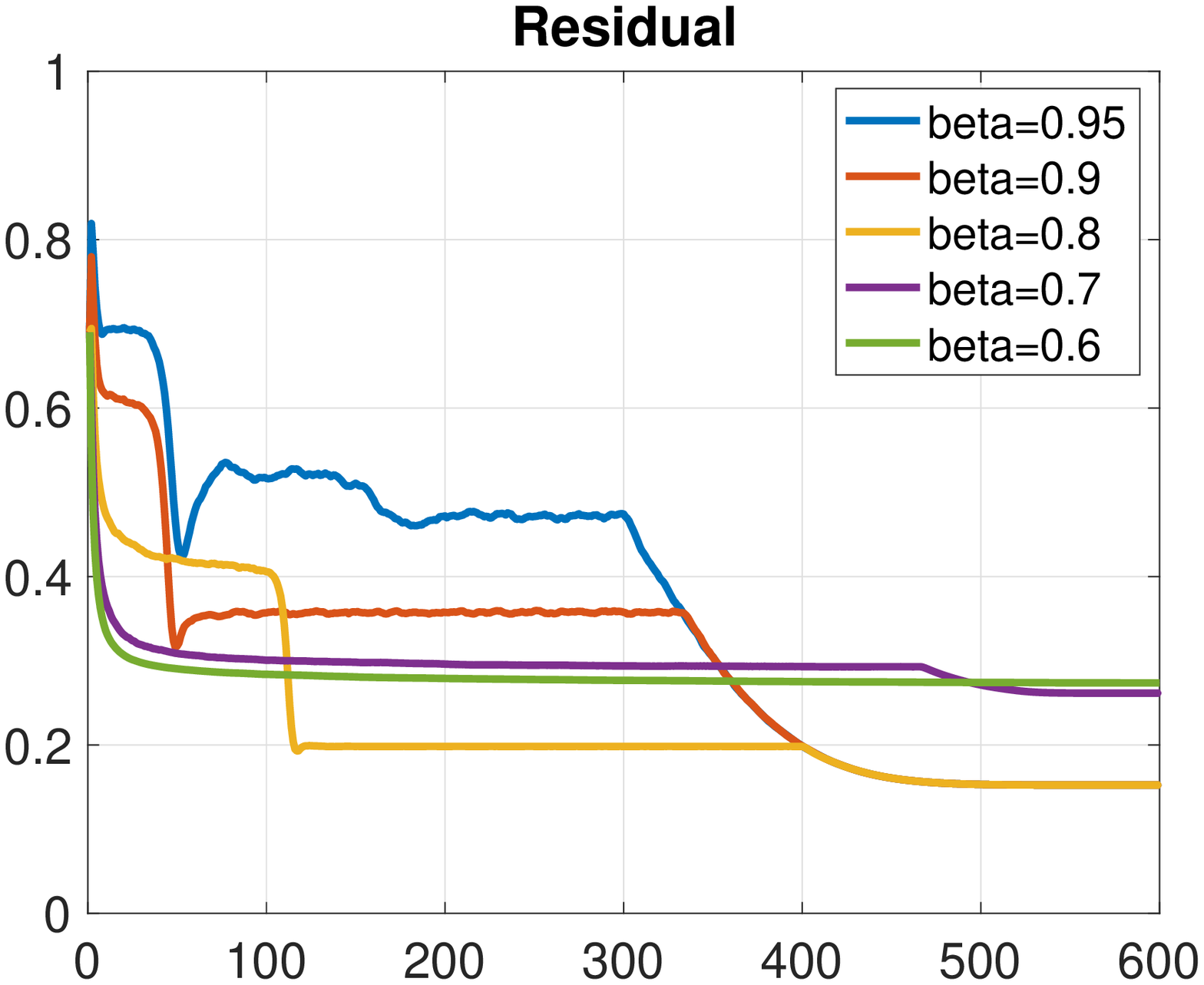}  \\
                        \includegraphics[width=0.45\textwidth]{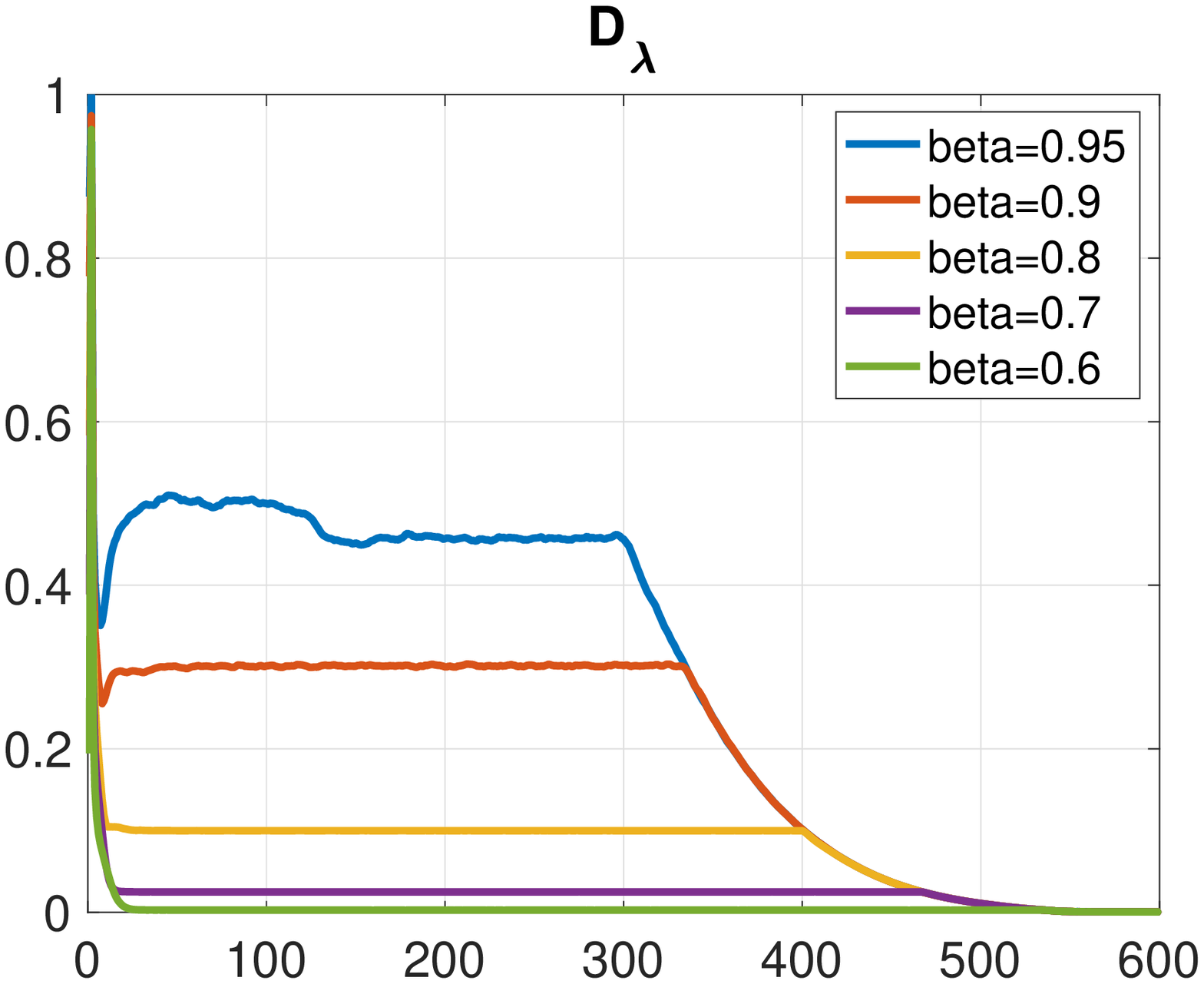}  
                        \includegraphics[width=0.45\textwidth]{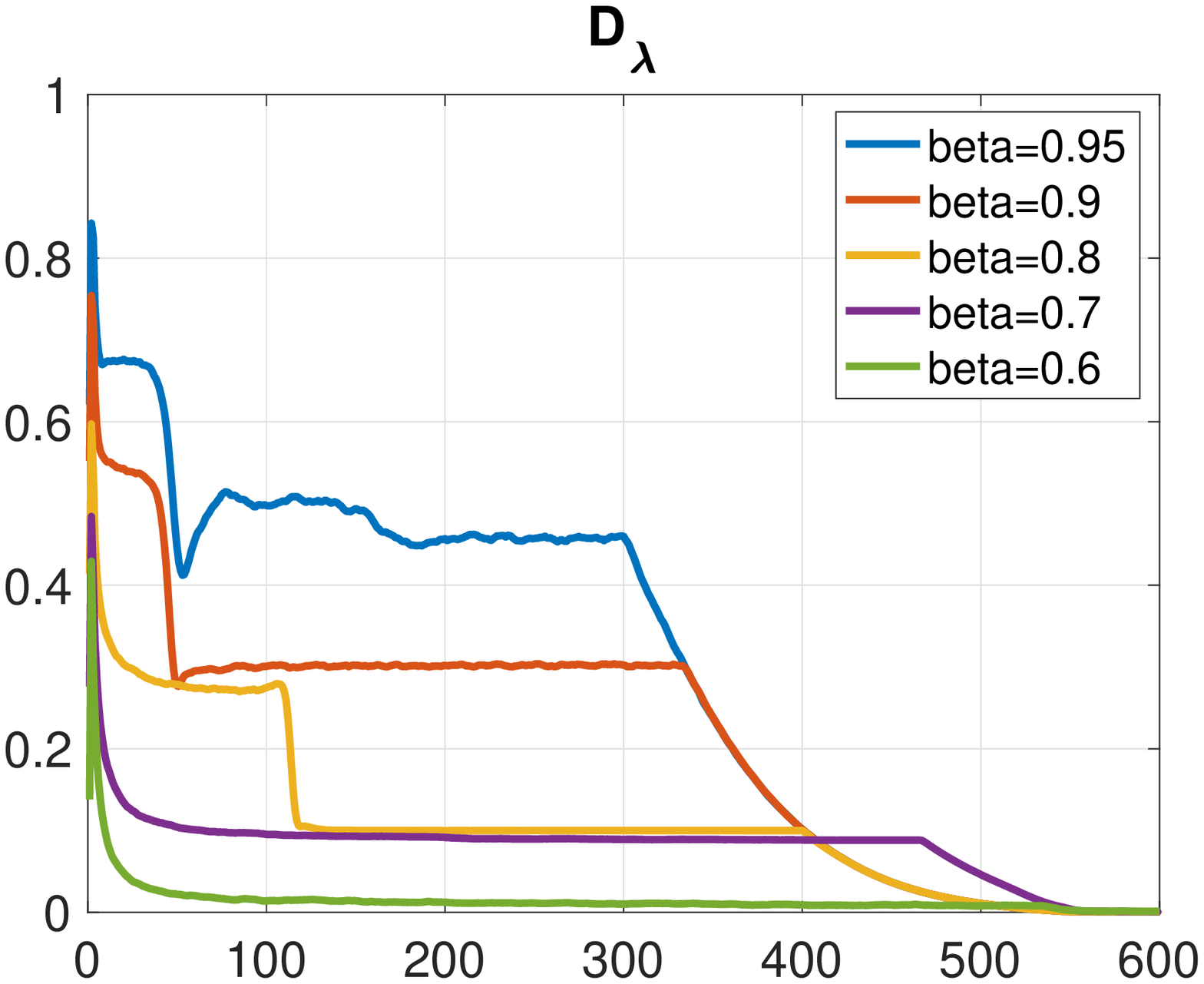} 
                         \end{center}
   \caption{ Left and right columns show the residue metric and the derivative metric of RAAR in    
    case(c) and case (d), respectively. 
  }  \label{expRAARcurves2new}
    \end{figure}
    \begin{figure}[htp]
 \begin{center}
                         \includegraphics[width=0.18\textwidth]{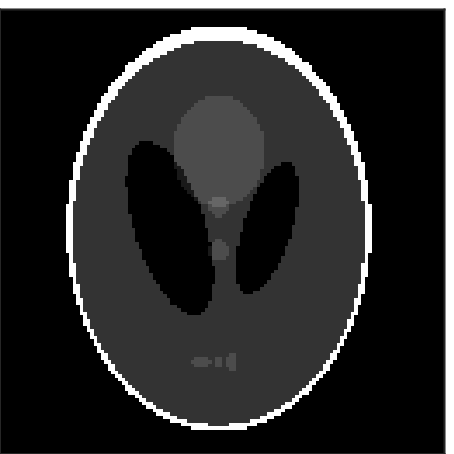}  
                        \includegraphics[width=0.18\textwidth]{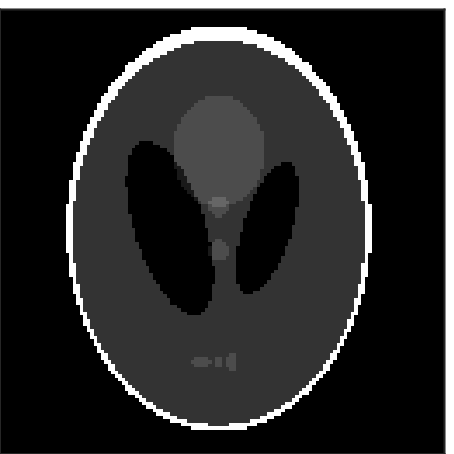}  
                        \includegraphics[width=0.18\textwidth]{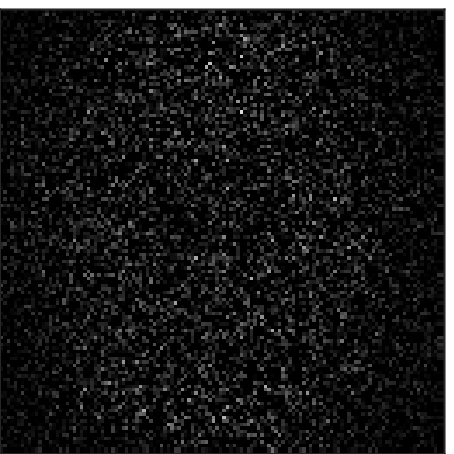}  
                        \includegraphics[width=0.18\textwidth]{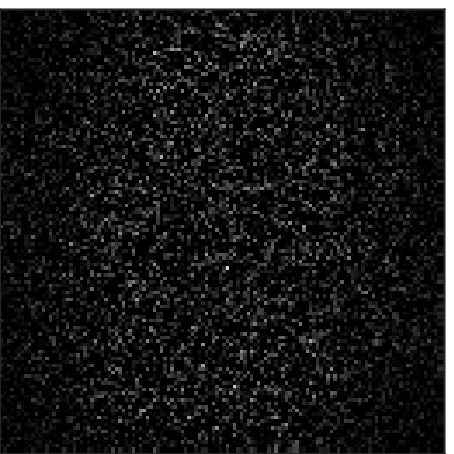}  
                        \includegraphics[width=0.18\textwidth]{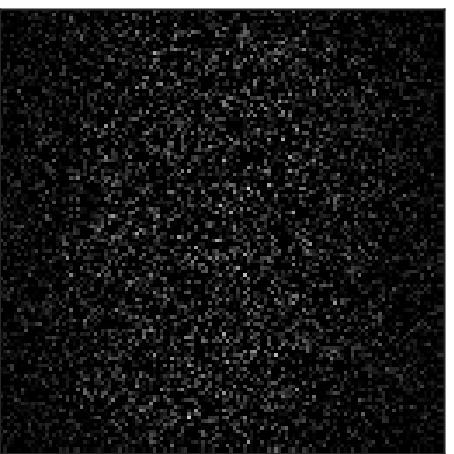}  
                     \end{center}
   \caption{  The  row from left to right show the reconstruction of 300 RAAR iterations in the case (b), corresponding to $\beta$-paths with $\beta=0.95$, $0.9$, $0.8$, $0.7$ and $0.6$, respectively.  }  \label{expRAARcurves40}
    \end{figure}

 \begin{figure}[htp]
 \begin{center}
                         \includegraphics[width=0.18\textwidth]{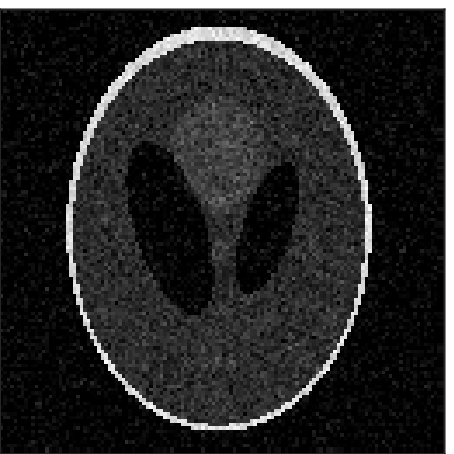}  
                        \includegraphics[width=0.18\textwidth]{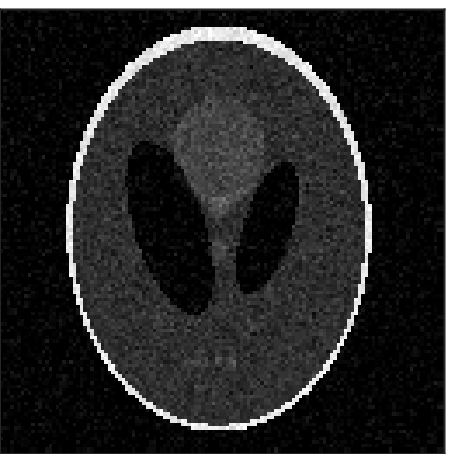}  
                        \includegraphics[width=0.18\textwidth]{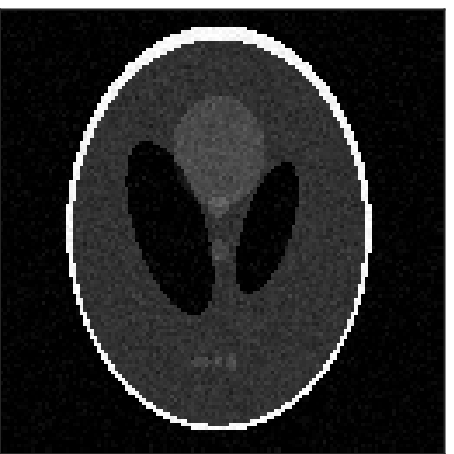}  
                        \includegraphics[width=0.18\textwidth]{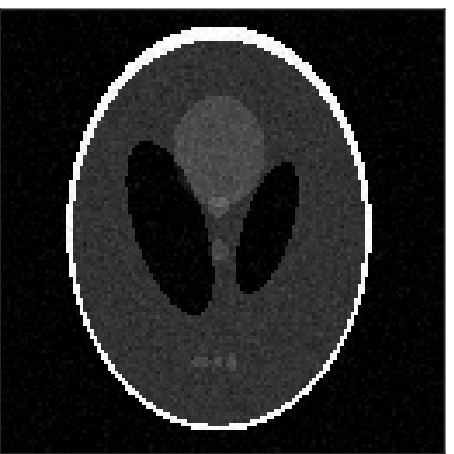}  
                        \includegraphics[width=0.18\textwidth]{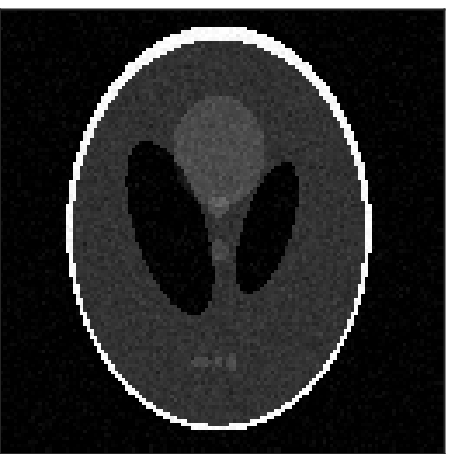}  \\
                        \includegraphics[width=0.18\textwidth]{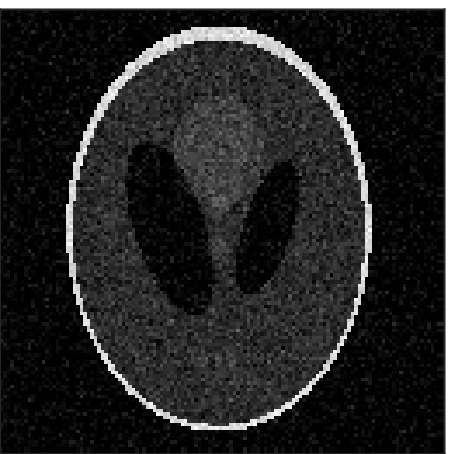}  
                        \includegraphics[width=0.18\textwidth]{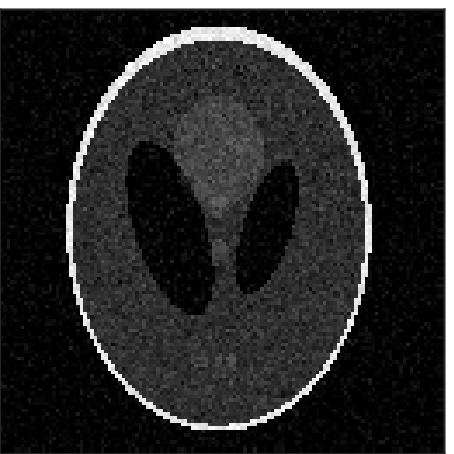}  
                        \includegraphics[width=0.18\textwidth]{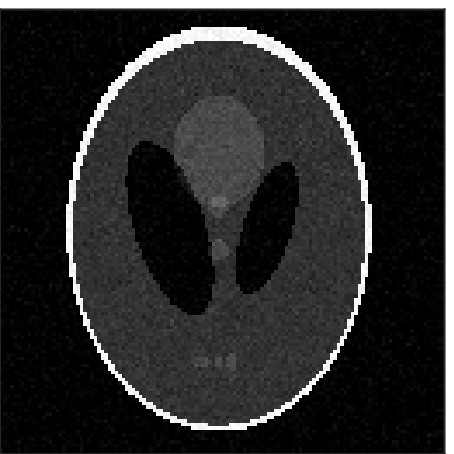}  
                        \includegraphics[width=0.18\textwidth]{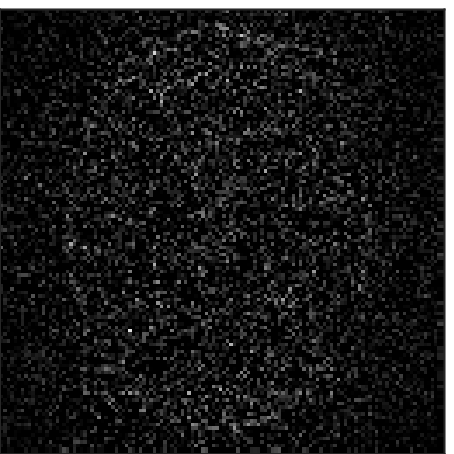}  
                        \includegraphics[width=0.18\textwidth]{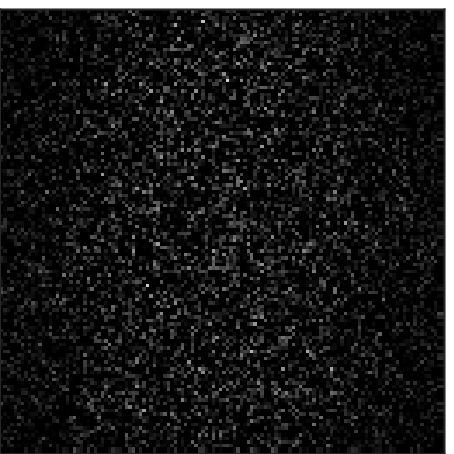}  
                      \end{center}
   \caption{ 
 The  columns  from left to right show the reconstruction of 300 RAAR iterations in the case (c,d), corresponding to $\beta$-paths with $\beta=0.95$, $0.9$, $0.8$, $0.7$ and $0.6$, respectively. The top row is the case (c) and the bottom row is the case (d). }  \label{expRAARcurves4}
    \end{figure}
 }

\subsection{Conclusion and outlook}
In this paper, we examine the RAAR convergence 
 from a viewpoint of  local saddles of a concave-non-convex max-min  problem.  We show that the global solution is a strictly local minimizer in oversampled coded diffraction patterns, which ensures  
 the existence of local saddles. Convergence to each local saddle of the RAAR  Lagrangian function  requires a sufficient large penalty parameter,  which  explains  the avoidance of undesired local solutions under RAAR with a moderate  penalty parameter. %{  In numerical studies,  we demonstrate the effectiveness of $\beta$-RAAR  on coherent diffraction patterns with $\beta$ traveling  from  large values to $0.5$.}

{ 
 ADMM is a popular algorithm in handling various constraints.  The current  paper does not introduce    any further assumption on  unknown objects,  except for the condition $x\in \IC^n$ in (\ref{main_P}). Stable recovery from incomplete measurements  is actually possible, provided that  additional assumptions of unknown objects are used.  For instance, 
when   unknown objects can be characterized by  piecewise-smooth functions with small total variation seminorm,
 the recovery can be obtained  from incomplete Fourier measurements with the aid of  total variation regularization~\cite{Chang2016}.  Another interesting  work 
~\cite{Eldar2014}  demonstrates   the number of measurements ensuring  stable recovery of a sparse object  under independent  measurement vectors.  From the above perspective,  one  interesting  future work is the
   saddle analysis of ADMM associated with  these additional object assumptions.
}
 
 \subsection{Acknowledgements}
 The author would like to thank Albert Fannjiang for  helpful discussions.
\bibliographystyle{unsrt} 

 \bibliography{RAARnew2}

\begin{thebibliography}{10}

\bibitem{Shechtman2015}
Y.~Shechtman, Y.~C. Eldar, O.~Cohen, H.~N. Chapman, J.~Miao, and M.~Segev.
\newblock Phase retrieval with application to optical imaging: a contemporary
  overview.
\newblock {\em Signal Processing Magazine, IEEE}, 32(3):87--109, 2015.

\bibitem{Chen2017}
P.~Chen, A.~Fannjiang, and G~Liu.
\newblock Phase retrieval with one or two coded diffraction patterns by
  alternating projection with the null initialization.
\newblock {\em Journal of Fourier Analysis and Applications}, pages 1--40,
  2017.

\bibitem{Gerchberg1972}
R.~W. Gerchberg.
\newblock A practical algorithm for the determination of phase from image and
  diffraction plane pictures.
\newblock {\em Optik}, 35:237, 1972.

\bibitem{NetrapalliEtAl2013Phase}
Praneeth Netrapalli, Prateek Jain, and Sujay Sanghavi.
\newblock Phase retrieval using alternating minimization.
\newblock {\em arxiv:1306.0160}, 2013.

\bibitem{ChenCandes2015Solving}
Yuxin Chen and Emmanuel~J. Candes.
\newblock Solving random quadratic systems of equations is nearly as easy as
  solving linear systems.
\newblock {\em arxiv:1505.05114}, 2015.

\bibitem{CandesEtAl2015Phase}
Emmanuel~J. Candes, Xiaodong Li, and Mahdi Soltanolkotabi.
\newblock Phase retrieval via wirtinger flow: Theory and algorithms.
\newblock {\em {IEEE} Transactions on Information Theory}, 61(4):1985--2007,
  apr 2015.

\bibitem{NIPS2016_83adc922}
Huishuai Zhang and Yingbin Liang.
\newblock Reshaped wirtinger flow for solving quadratic system of equations.
\newblock In D.~Lee, M.~Sugiyama, U.~Luxburg, I.~Guyon, and R.~Garnett,
  editors, {\em Advances in Neural Information Processing Systems}, volume~29,
  pages 2622--2630. Curran Associates, Inc., 2016.

\bibitem{Chen2017a}
P.~Chen, A.~Fannjiang, and G~Liu.
\newblock Phase retrieval by linear algebra.
\newblock {\em SIAM J. Matrix Anal. appl.}, 38(3):864--868, 2017.

\bibitem{LuoEtAl2018Optimal}
Wangyu Luo, Wael Alghamdi, and Yue~M. Lu.
\newblock Optimal spectral initialization for signal recovery with applications
  to phase retrieval.
\newblock {\em arXiv:1811.04420}, 2018.

\bibitem{LuLi2017Phase}
Yue~M. Lu and Gen Li.
\newblock Phase transitions of spectral initialization for high-dimensional
  nonconvex estimation.
\newblock {\em arxiv:1702.06435}, 2017.

\bibitem{Mondelli2019}
Marco Mondelli and Andrea Montanari.
\newblock Fundamental limits of weak recovery with applications to phase
  retrieval.
\newblock {\em Foundations of Computational Mathematics}, 19(3):703--773, Jun
  2019.

\bibitem{duchi2019solving}
John~C Duchi and Feng Ruan.
\newblock Solving (most) of a set of quadratic equalities: Composite
  optimization for robust phase retrieval.
\newblock {\em Information and Inference: A Journal of the IMA}, 8(3):471--529,
  2019.

\bibitem{Fienup1982}
J.~R. Fienup.
\newblock Phase retrieval algorithms: a comparison.
\newblock {\em Applied optics}, 21(15):2758--2769, 1982.

\bibitem{Fienup2013}
J.~R. Fienup.
\newblock Phase retrieval algorithms: a personal tour.
\newblock {\em Applied Optics}, 52(1):45--56, 2013.

\bibitem{Bauschke2003}
Heinz~H. Bauschke, Patrick~L. Combettes, and D.~Russell Luke.
\newblock Hybrid projection--reflection method for phase retrieval.
\newblock {\em J. Opt. Soc. Am. A}, 20(6):1025--1034, Jun 2003.

\bibitem{CHEN2018665}
Pengwen Chen and Albert Fannjiang.
\newblock Fourier phase retrieval with a single mask by {D}ouglas–{R}achford
  algorithms.
\newblock {\em Applied and Computational Harmonic Analysis}, 44(3):665 -- 699,
  2018.

\bibitem{Wen_2012}
Zaiwen Wen, Chao Yang, Xin Liu, and Stefano Marchesini.
\newblock Alternating direction methods for classical and ptychographic phase
  retrieval.
\newblock {\em Inverse Problems}, 28(11):115010, oct 2012.

\bibitem{Luke_2004}
D~Russell Luke.
\newblock Relaxed averaged alternating reflections for diffraction imaging.
\newblock {\em Inverse Problems}, 21(1):37--50, nov 2004.

\bibitem{Eckstein1992}
Jonathan Eckstein and Dimitri~P. Bertsekas.
\newblock On the {D}ouglas---{R}achford splitting method and the proximal point
  algorithm for maximal monotone operators.
\newblock {\em Mathematical Programming}, 55(1):293--318, Apr 1992.

\bibitem{He2015}
Bingsheng He and Xiaoming Yuan.
\newblock On the convergence rate of {D}ouglas-{R}achford operator splitting
  method.
\newblock {\em Mathematical Programming}, 153(2):715--722, 2015.

\bibitem{Li}
Ji~Li and Tie Zhou.
\newblock On relaxed averaged alternating reflections ({RAAR}) algorithm for
  phase retrieval with structured illumination.
\newblock {\em Inverse Problems}, 33(2):025012, jan 2017.

\bibitem{Bertsekas1996}
Dimitri~P. Bertsekas.
\newblock {\em Constrained Optimization and Lagrange Multiplier Methods
  (Optimization and Neural Computation Series)}.
\newblock Athena Scientific, 1 edition, 1996.

\bibitem{Hong2016}
Mingyi Hong, Zhi-Quan Luo, and Meisam Razaviyayn.
\newblock Convergence analysis of alternating direction method of multipliers
  for a family of nonconvex problems.
\newblock {\em SIAM Journal on Optimization}, 26(1):337--364, 2016.

\bibitem{li2015global}
Guoyin Li and Ting~Kei Pong.
\newblock Global convergence of splitting methods for nonconvex composite
  optimization.
\newblock {\em SIAM Journal on Optimization}, 25(4):2434--2460, 2015.

\bibitem{wang2019global}
Yu~Wang, Wotao Yin, and Jinshan Zeng.
\newblock Global convergence of {ADMM} in nonconvex nonsmooth optimization.
\newblock {\em Journal of Scientific Computing}, 78(1):29--63, 2019.

\bibitem{Bert}
D.P. Bertsekas.
\newblock {\em Constrained Optimization and Lagrange Multiplier Methods}.
\newblock Computer science and applied mathematics. Elsevier Science, 2014.

\bibitem{Boyd}
Stephen Boyd, Neal Parikh, Eric Chu, Borja Peleato, and Jonathan Eckstein.
\newblock Distributed optimization and statistical learning via the alternating
  direction method of multipliers.
\newblock {\em Found. Trends Mach. Learn.}, 3(1):1--122, January 2011.

\bibitem{Sun}
J.~{Sun}, Q.~{Qu}, and J.~{Wright}.
\newblock A geometric analysis of phase retrieval.
\newblock In {\em 2016 IEEE International Symposium on Information Theory
  (ISIT)}, pages 2379--2383, July 2016.

\bibitem{Lee2016}
Jason~D. Lee, Max Simchowitz, Michael~I. Jordan, and Benjamin Recht.
\newblock Gradient descent only converges to minimizers.
\newblock In Vitaly Feldman, Alexander Rakhlin, and Ohad Shamir, editors, {\em
  29th Annual Conference on Learning Theory}, volume~49 of {\em Proceedings of
  Machine Learning Research}, pages 1246--1257, Columbia University, New York,
  New York, USA, 23--26 Jun 2016. PMLR.

\bibitem{NIPS2017_f79921bb}
Simon~S Du, Chi Jin, Jason~D Lee, Michael~I Jordan, Aarti Singh, and Barnabas
  Poczos.
\newblock Gradient descent can take exponential time to escape saddle points.
\newblock In I.~Guyon, U.~V. Luxburg, S.~Bengio, H.~Wallach, R.~Fergus,
  S.~Vishwanathan, and R.~Garnett, editors, {\em Advances in Neural Information
  Processing Systems}, volume~30, pages 1067--1077. Curran Associates, Inc.,
  2017.

\bibitem{jin2017escape}
Chi Jin, Rong Ge, Praneeth Netrapalli, Sham~M Kakade, and Michael~I Jordan.
\newblock How to escape saddle points efficiently.
\newblock In {\em 34th International Conference on Machine Learning, ICML
  2017}, pages 2727--2752. International Machine Learning Society (IMLS), 2017.

\bibitem{goodfellow2020generative}
Ian Goodfellow, Jean Pouget-Abadie, Mehdi Mirza, Bing Xu, David Warde-Farley,
  Sherjil Ozair, Aaron Courville, and Yoshua Bengio.
\newblock Generative adversarial networks.
\newblock {\em Communications of the ACM}, 63(11):139--144, 2020.

\bibitem{omidshafiei2017deep}
Shayegan Omidshafiei, Jason Pazis, Christopher Amato, Jonathan~P How, and John
  Vian.
\newblock Deep decentralized multi-task multi-agent reinforcement learning
  under partial observability.
\newblock {\em arXiv preprint arXiv:1703.06182}, 2017.

\bibitem{Adolphs2018}
Leonard Adolphs, Hadi Daneshmand, Aurelien Lucchi, and Thomas Hofmann.
\newblock Local saddle point optimization: A curvature exploitation approach.
\newblock In {\em The 22nd International Conference on Artificial Intelligence
  and Statistics}, pages 486--495. PMLR, 2019.

\bibitem{Daskalakis2018}
Constantinos Daskalakis and Ioannis Panageas.
\newblock The limit points of (optimistic) gradient descent in min-max
  optimization.
\newblock {\em Advances in Neural Information Processing Systems},
  31:9236--9246, 2018.

\bibitem{Jin2020}
Chi Jin, Praneeth Netrapalli, and Michael Jordan.
\newblock What is local optimality in nonconvex-nonconcave minimax
  optimization?
\newblock In {\em International Conference on Machine Learning}, pages
  4880--4889. PMLR, 2020.

\bibitem{Dai2020}
Yu-Hong Dai and Liwei Zhang.
\newblock Optimality conditions for constrained minimax optimization.
\newblock {\em CSIAM Transactions on Applied Mathematics}, 1(2):296--315, 2020.

\bibitem{Fannjiang_2012}
Albert Fannjiang.
\newblock Absolute uniqueness of phase retrieval with random illumination.
\newblock {\em Inverse Problems}, 28(7):075008, jun 2012.

\bibitem{fannjiang2020fixed}
Albert Fannjiang and Zheqing Zhang.
\newblock Fixed point analysis of douglas--rachford splitting for ptychography
  and phase retrieval.
\newblock {\em SIAM Journal on Imaging Sciences}, 13(2):609--650, 2020.

\bibitem{Glowinski1975}
R.~Glowinski and A.~Marroco.
\newblock Sur l'approximation, par \'el\'ements finis d'ordre un, et la
  r\'esolution, par p\'enalisation-dualit\'e d'une classe de probl\`emes de
  dirichlet non lin\'eaires.
\newblock {\em ESAIM: Mathematical Modelling and Numerical Analysis -
  Mod\'elisation Math\'ematique et Analyse Num\'erique}, 9(R2):41--76, 1975.

\bibitem{Gabay1976}
Daniel Gabay and Bertrand Mercier.
\newblock A dual algorithm for the solution of nonlinear variational problems
  via finite element approximation.
\newblock {\em Computers Mathematics with Applications}, 2(1):17–40, 1976.

\bibitem{yang2009}
Junfeng Yang, Yin Zhang, and Wotao Yin.
\newblock An efficient tvl1 algorithm for deblurring multichannel images
  corrupted by impulsive noise.
\newblock {\em SIAM Journal on Scientific Computing}, 31(4):2842--2865, 2009.

\bibitem{Deka2019}
Bhabesh Deka and Sumit Datta.
\newblock Compressed sensing magnetic resonance image reconstruction
  algorithms.
\newblock {\em Springer Series on Bio-and Neurosystems}, 2019.

\bibitem{doi:10.1137/15M1026924}
A.~Cherukuri, B.~Gharesifard, and J.~Cortés.
\newblock Saddle-point dynamics: Conditions for asymptotic stability of saddle
  points.
\newblock {\em SIAM Journal on Control and Optimization}, 55(1):486--511, 2017.

\bibitem{Thibault2012}
P~Thibault and M~Guizar-Sicairos.
\newblock Maximum-likelihood refinement for coherent diffractive imaging.
\newblock {\em New Journal of Physics}, 14(6):063004, jun 2012.

\bibitem{Chang2016}
Huibin Chang, Yifei Lou, Michael~K. Ng, and Tieyong Zeng.
\newblock Phase retrieval from incomplete magnitude information via total
  variation regularization.
\newblock {\em SIAM J Sci Comput}, 38(6):A3672--A3695, 2016.

\bibitem{Eldar2014}
Y.~C. Eldar and S.~Mendelson.
\newblock Phase retrieval: Stability and recovery guarantees.
\newblock {\em Applied and Computational Harmonic Analysis}, 36(3):473--494,
  2014.

\end{thebibliography}
\end{document}